\documentclass[11pt,a4paper]{article}
\usepackage{amsthm}
\usepackage{amsfonts, amssymb, amsmath, amscd, latexsym}
\usepackage{eucal, enumerate, latexsym}
\usepackage{graphics}
\usepackage{graphicx}

\usepackage{pict2e}

\newcommand{\R}{\mathbb R}
\newcommand{\N}{\mathbb N}
\newtheorem{theorem}{Theorem}[section]
\newtheorem{corollary}{Corollary}[section]
\newtheorem{proposition}{Proposition}[section]
\newtheorem{definition}{Definition}[section]
\newtheorem{remark}{Remark}[section]
\newtheorem{lemma}{Lemma}[section]

\newcommand{\orho}{\overline{\rho}}
\newcommand{\oxi}{\overline{\xi}}
\newcommand{\ds}{\displaystyle}

\renewcommand{\epsilon}{\varepsilon}
\newcommand{\eps}{\epsilon}
\renewcommand{\phi}{\varphi}
\newenvironment{proofof}[1]{\smallskip\noindent\emph{Proof of #1.}%
\hspace{1pt}}{\hspace{-5pt}{\nobreak\quad\nobreak\hfill\nobreak%
$\square$\vspace{8pt}\par}\smallskip\goodbreak}

\makeatletter

\newlength{\captionwidth}
\long\def\@makecaption#1#2{%
   \vskip 10\p@
   \setbox\@tempboxa\hbox{#1: #2}%
   \ifdim \wd\@tempboxa > \captionwidth 
       \hbox to\hsize{\hfil
       \parbox[t]{\captionwidth}{
       \small#1: \small#2\par}
       \hfil}
     \else
       \hbox to\hsize{\hfil\box\@tempboxa\hfil}%
   \fi}

\makeatother

\begin{document}

\title{Sharp profiles in models of collective movements}

\author{Andrea Corli\\
\small\textit{Department of Mathematics and Computer Science, University of Ferrara}\\
\small\textit{I-44121 Italy, e-mail: andrea.corli@unife.it}, \\\\
Lorenzo di Ruvo\\
\small\textit{Department of Sciences and Methods for Engineering, University of Modena and Reggio Emilia}\\
\small\textit{I-42122 Italy, e-mail: lorenzo.diruvo@unimore.it},
\and Luisa Malaguti\\
\small\textit{Department of Sciences and Methods for Engineering, University of Modena and Reggio Emilia}\\
\small\textit{I-42122 Italy, e-mail: luisa.malaguti@unimore.it}}

\maketitle

\par\vspace*{-.03\textheight}{\center
{\emph{Dedicated to Alberto Bressan on the occasion of his 60th birthday}}\par}

\begin{abstract}
We consider a parabolic partial differential equation that can be understood as a simple model for crowds flows. Our main assumption is that the diffusivity and the source/sink term vanish at the same point; the nonhomogeneous term is different from zero at any other point and so the equation is not monostable. We investigate the existence, regularity and monotone properties of semi-wavefront solutions as well as their convergence to wavefront solutions. 

\vspace{1cm}

\end{abstract}

\maketitle
\section{Introduction}\label{s:I}
In this paper we consider the scalar advection-reaction-diffusion equation
\begin{equation}\label{e:E}
\rho_t + f(\rho)_x=\left(D(\rho)\rho_x\right)_x+g(\rho), \qquad t\ge 0, \, x\in \R,
\end{equation}
for the unknown function $\rho=\rho(x,t)$. We assume that $f\in C^1[0,\overline \rho]$, $f(0)=0$, $g\in C[0,\overline \rho]$, $D\in C^1[0,\overline \rho]$ and denote for short $h(\rho):=f'(\rho)$; here, $\orho$ is a positive constant. All along the paper we consider several different conditions on both $D$ and $g$ but we mainly focus on the case that $D$ satisfies 
\begin{itemize}

\item[{(D)}] \, $D(\orho)=0$ and $D(\rho)>0$ for $\rho\in(0, \overline \rho)$.

\end{itemize}
About the forcing term $g$ we mainly deal with the following assumption:
\begin{itemize}
\item[{(g)}] \, $g(\rho)>0$ for $\rho\in[0, \overline \rho)$ and $g(\overline \rho)=0$.
\end{itemize}
We refer to Figure \ref{fig:D} for a pictorial summary showing the typical behavior of the diffusivity $D$ and of the source/sink term $g$ according to the assumptions we require here and below.


\begin{figure}[htbp]
\begin{picture}(100,120)(-80,-10)
\setlength{\unitlength}{1pt}

\put(160,0)
{
\put(-220,0){
\put(0,0){\vector(1,0){180}}
\put(180,-5){\makebox(0,0)[t]{$\rho$}}
\put(0,0){\vector(0,1){100}}
\put(-5,100){\makebox(0,0)[r]{$D$}}
\put(150,-5){\makebox(0,0)[t]{$\overline{\rho}$}}

\put(0,0){\thicklines{\qbezier(0,0)(75,80)(150,0)}} 
\put(18,10){\makebox(0,0)[l]{(${\rm D}$)}}

\put(0,0){\thicklines{\qbezier(0,60)(105,58)(150,0)}} 
\put(18,55){\makebox(0,0)[t]{(D)}}

\put(0,0){\thicklines{\qbezier(0,72.5)(150,58)(150,0)}} 
\put(152,15){\makebox(0,0)[l]{($\tilde{\rm D}$)}}

\put(0,0){\thicklines{\qbezier(0,85)(105,80)(150,30)}} 
\put(18,88){\makebox(0,0)[b]{$(\hat{\rm D})$}}

}

\put(0,30){
\put(0,0){\vector(1,0){180}}
\put(180,-5){\makebox(0,0)[t]{$\rho$}}
\put(0,0){\vector(0,1){70}}
\put(-5,70){\makebox(0,0)[r]{$g$}}
\put(150,6){\makebox(0,0)[b]{$\overline{\rho}$}}

\put(0,0){\thicklines{\qbezier(0,60)(105,58)(150,0)}} 
\put(18,65){\makebox(0,0)[b]{(g)}}

\put(0,0){\thicklines{\qbezier(0,0)(85,58)(150,0)}} 
\put(12,26){\makebox(0,0)[t]{(${\rm g}_0$)}}

\put(0,0){\thicklines{\qbezier(0,0)(105,-38)(150,-40)}} 
\put(30,-30){\makebox(0,0)[b]{$(\hat{\rm g})$}}

\put(0,0){\thicklines{\qbezier(0,40)(40,35)(75,0)}} 
\put(0,0){\thicklines{\qbezier(75,0)(100,-20)(150,-20)}} 
\put(140,-16){\makebox(0,0)[b]{(g${}_1$)}}

}
}\end{picture}
\caption{\label{fig:D}{Some examples of diffusivities $D$ and source/sink terms $g$ used in the paper.}}
\end{figure}
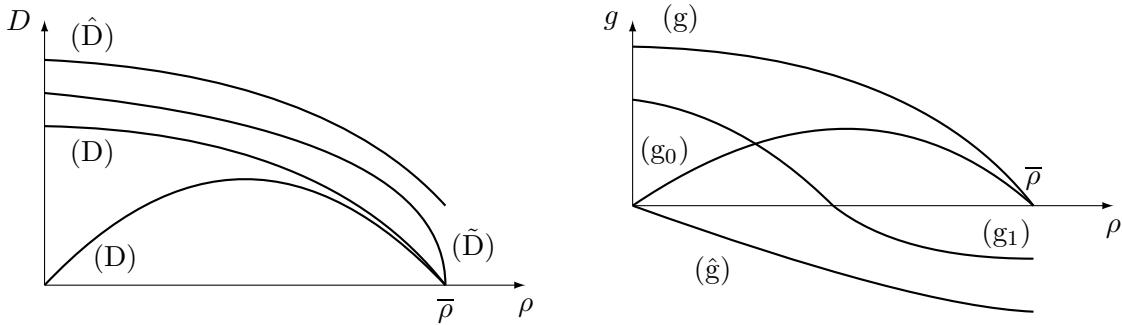

Equation \eqref{e:E} was recently proposed in \cite{BTTV, Corli-Malaguti} as a simplified model for collective movements in one spatial dimension. In that case the function $\rho(x,t)$ represents the density of pedestrians at point $x$ and time $t$, while $\orho$ is the maximal density that the environment can support. Pedestrians are advected by the function $f$; a common choice in this framework is $f(\rho)=\rho v(\rho)$, where $v$ is an assigned density-dependent velocity. The term $D$ accounts for diffusion effects and takes into account the visual depth; under the above choice of $f$, in \cite{BTTV} the authors deduce $D(\rho)=-\delta \rho v'(\rho)$ by a Chapman-Enskog expansion, where $\delta$ is the visual depth. At last, a source term $g$ satisfying (g) represents entries, which are thought to be diffused rather than localized at some place \cite{Bagnerini-Colombo-Corli}. This modeling can be meaningful, for instance, for a crowd flowing along a corridor with many side entries; think at the barrier of a subway exit or at the platforms of a railway station reaching the main hall: instead of modeling each single entry we use a continuum description. If $g$ is a decreasing function of $\rho$, then entries are high at low densities on the line, low at high densities and are blocked when the maximal density $\orho$ is reached. We refer to \cite{Corli-Malaguti} for a more detailed discussion of this model as well as to other physical and biological phenomena that equation \eqref{e:E} models.

We are interested in the existence and regularity of constant-profile solutions $\rho$ of \eqref{e:E}, i.e. solutions of the form $\rho(x,t)=\varphi(x-ct)$, where $\varphi(\xi)$ is the wave profile and $c$ the constant wave speed. In this case $\varphi(\xi)$ satisfies the equation
\begin{equation}\label{e:tws}
\left(D(\varphi)\varphi^{\, \prime}\right)^{\, \prime} + \left(c-h(\varphi)\right)\varphi^{\, \prime}+g(\varphi)=0
\end{equation}
in some open interval $I\subseteq \R$, where ${}^{\, \prime}$ denotes differentiation with respect to $\xi$. We refer to \cite{Bonheure-Sanchez, GK} for a wide treatment of this topic.

In \cite{Corli-Malaguti} we discussed in detail this topic in the case the diffusivity $D$ satisfies both $D(0)=0$ and $D(\orho)>0$; these conditions were deduced by experimental data in \cite{BTTV} but we treated as well the case $D(0)>0$. We proved that for every wave speed $c$ there exist {\em classical} semi-wavefront solutions, i.e. constant-profile solutions such that their profile $\phi$ is defined in a half-line and $D(\varphi)\varphi^{\, \prime}$ is absolutely continuous. We also showed that the slope of the profiles $\phi$ at points where $\rho=0$ depends on the order of vanishing of $D$ at $0$. At last, we proved that pasting semi-wavefront solutions {\em never} yields {\em global} traveling-waves.

In the current paper, aiming at the widest generality, we make no assumption about the vanishing of $D$ at $0$; however, we mostly assume $D(\orho)=0$. Some motivations to this latter assumption can be found in \cite{Bonzani2000, Bonzani-Mussone2002, De_Angelis1999}; a naive explanation, in terms of the above model, is the following. Assume $f(\rho) = \rho v(\rho)$ as above; then, it is natural to require $v(\orho)=0$. If $D(\orho)>0$, then the effect of diffusion is to let vehicles or pedestrians move {\em backwards} at points where the maximal density is reached, because $v(\orho)=0$; this is in contradiction with the phenomenon we are modeling.

The degenerate behavior induced by assumption (D) does not affect most of the existence results given in \cite{Corli-Malaguti}; however, here we show that it causes a lack of their regularity and leads to {\em sharp} semi-wavefront solutions \cite{SanchezGarduno-Maini}. This is in contrast to \cite{Corli-Malaguti}, where only {\em classical} solutions appeared, and is due to the vanishing of {\em both} $D$ {\em and} $g$ at $\orho$.  We refer to Definition \ref{d:tws} below for both definitions. Now, we provide a detailed account on these sharp solutions.

A sharp semi-wavefront solution is constant on a half-plane and reaches this value in a continuous but non-differentiable way. The interest in sharp solutions is related to the important property of finite speed of propagation, as showed in \cite{Gilding-Kersner_1996}. The appearance of sharp profiles was first discussed in \cite{Aronson} within the framework of models of biological invasion. The equation studied there is
\begin{equation}\label{e:Aronson}
u_t=(u^m)_{xx}+u(1-u),
\end{equation}
with $m>1$;  it is showed that equation \eqref{e:Aronson} supports wavefronts for a {\em half-line} of admissible speeds and the solution with minimal speed $c^*(m)$ has a sharp behavior. The explicit computation of this sharp solution is provided in \cite{Murray} when $m=2$; in that case $c^*(2)=1/{\sqrt 2}$. The case when the source term $u(1-u)$ in \eqref{e:Aronson} is replaced by $u(1-u)(u-\alpha)$, with $\alpha \in (0, \frac 12)$, is treated in \cite{Hosono} and again a sharp wavefront arises. Equation \eqref{e:Aronson} led to study the general model
\begin{equation}\label{e:Maini}
u_t=\left(D(u)u_x\right)_{x}+g(u),
\end{equation}
with $D(0)=0$, in the monostable case
\begin{equation}\label{e:Fisher}
g(u)>0 \text{ in } (0,1) \text{ and } g(0)=g(1)=0.
\end{equation}
The uniqueness of a sharp solution  (in the class of classical or sharp solutions) is proved in \cite{SanchezGarduno-Maini}. A rather general discussion on sharp wavefronts appeared in \cite{MMconv} for the equation
\begin{equation}\label{e:conv}
u_t+h(u)u_x = \left(D(u)u_x\right)_{x}+g(u),
\end{equation}
which incorporates the convective term $h(u)$; the source term $g$  satisfies again \eqref{e:Fisher} and $D$ has essentially a polynomial growth near $0$ and $1$ if it also vanishes there. A doubly sharp behavior can be observed in the latter case. In recent years, sharp profiles have been considered in equations with delay arguments in the source term \cite{Jin} and in coupled equations  see \cite{Ducrot, Satnoianu}. We also quote \cite{Duijin} for a model without source term but with the presence of an extra term depending on $u_{xt}$.

To the best of our knowledge, the study of sharp profiles done in this paper is new and our results can be derived by none of the papers quoted above. The reason is twofold. On the one hand, here we have no restrictions about the growth of $D$ near $\orho$. On the other hand, we deal with semi-wavefronts corresponding to {\em every} wave speed; on the contrary, only a {\em half-line} of wavefront speeds is admissible for equation \eqref{e:conv} under condition \eqref{e:Fisher}. We notice that, as in the monostable case, a critical threshold $c^*$ appears when (g) holds \cite{Corli-Malaguti}, and satisfies completely analogous estimates. However, while in the former case $c^*$ separates the existence or the failure of wavefronts, in the latter it only gives information about the slope of the profiles when they reach the value $0$.

The plan of the paper now follows. In Section \ref{s:main} we give precise definitions and state our results. We establish the existence of semi-wavefront solutions for {\em every} wave speed $c$, we characterize the occurrence of classical or sharp profiles and show some monotonicity properties. Indeed, we also deal with the case when $D \in C[0, \orho]\cap C^1[0, \orho)$ satisfies
\begin{itemize}

\item[{(${\rm \tilde D}$)}] \, $D(\orho)=0$, $D(\rho)>0$ for $\rho\in(0,\orho)$, $\dot D(\orho) = -\infty$; moreover, there exists $\lim_{\phi\to\orho^-}\frac{D(\phi)g(\phi)}{\phi-\orho}\in(-\infty,0]
$.

\end{itemize}
Then, we briefly consider the companion case when the source term $g$ models exists \cite{Bagnerini-Colombo-Corli} instead of entries. This case (even under the assumption $D(\orho)>0$) was not treated in \cite{Corli-Malaguti}; in order to cover also that case, we make no requirements about the vanishing of $D$ at $\orho$ and simply assume
\begin{itemize}
\item[{($\hat {\rm D}$)}] \, $D(\rho)>0$ for $\rho\in(0, \orho)$;
\item[{$(\hat{\rm g})$}] \, $g(\rho)<0$ for $\rho\in(0, \overline \rho]$ and $g(0)=0$.
\end{itemize}
We also refer to \cite{Kiselev-Ryzhik_II, Kiselev-Ryzhik_I} for a source term satisfying condition {$(\hat{\rm g})$} in a different framework.  

Next, we deal with the convergence of semi-wavefronts to wavefronts under the assumption that the diffusivity still satisfies (${\rm \hat D}$) or (${\rm \tilde D}$). We consider a decreasing sequence of source terms $g_n$ that satisfy (g) and converge uniformly to a source term $g_0\in C[0,\orho]$ satisfying the monostability condition, see \eqref{e:Fisher},
\begin{itemize}
\item[{(${\rm g}_0$)}] \, $g(\rho)>0$ for $\rho\in(0, \overline \rho), \, \, g(0)=g(\overline \rho)=0$ and $\displaystyle{\limsup_{\rho \to 0^+}} \frac{D(\rho)g(\rho)}{\rho}<+\infty$.
\end{itemize}
Notice that, when $D(0)=0$, the last condition in (${\rm g}_0$) is automatically satisfied.  As we pointed out above, the equation associated with $g_0$ admits wavefronts and the issue is whether and how the semi-wavefront profiles $\phi_n$ associated to the equation with $g_n$ converge to the profile $\phi_0$ associated to the equation with $g_0$. We thank C. Mascia for having risen such a question. Notice that assumptions (${\rm \tilde D}$) and (${\rm g}_0$) mix together $D$ and $g$; indeed, this mixing is well known when dealing with diffusivities and source terms that may vanish at the same point. 

Our last result concerns the case when the term $g$ changes sign; namely, we assume
\begin{itemize}
\item[{(${\rm g}_1$)}] \, $g(\rho)>0$ for $\rho\in[0, \rho_0)$ and $g(\rho)<0$ for   $\rho\in(\rho_0, \overline \rho]$,
\end{itemize}
with $\rho_0\in(0,\orho)$. Such a term may be thought to model entries if $\rho\in[0,\rho_0)$ and exits if $\rho\in(\rho_0,\orho]$. Under a further local assumption at $\rho_0$ we prove the existence of several patterns of traveling waves.

The results provided in Section \ref{s:main} do not cover all possible cases: we tried to deal with the most significant situations while avoiding exceedingly complicated statements and proofs. However, it is not difficult to extend our results by a suitable mixing of the techniques exploited in \cite{Corli-Malaguti} and in the current paper.

Section \ref{s:first solution} is concerned with a technical tool that was intensively used in \cite{Corli-Malaguti}; namely, the reduction of the second-order equation \eqref{e:tws} to the {\em singular} first-order equation
\begin{equation}\label{e:zeq}
\dot z(\varphi)=h(\varphi)-c-\frac{D(\varphi)g(\varphi)}{z(\varphi)}, \qquad \varphi \in (0, \overline \rho).
\end{equation}
Such an order reduction depends on the strict monotonicity of the wave profile  $\varphi$ in the interval where $0 \le \phi(\xi) < \orho$. In that case, if we denote by $\xi=\xi(\phi)$ the inverse function of $\phi$, then the function $z$ is defined by $z(\varphi):=D(\varphi)\varphi^{\, \prime}\left(\xi(\varphi)\right)$ for $\varphi\in (0, \overline \rho)$. We point out that similar techniques were recently exploited in \cite{Garrione-Strani} in the case $D$ is a saturating diffusion depending on $\rho_x$ instead of $\rho$.

In Section \ref{s:nature} we first prove a property of semi-wavefront profiles and deduce that our definition of sharp profiles is essentially equivalent to a previous one given in \cite{MMconv}; that section and Sections \ref{s:nature-1} contain the proofs of our main results. As in \cite[\S 8]{Corli-Malaguti}, the procedure of pasting semi-wavefront solutions to obtain a global traveling wave is unsuccessful.  Sections \ref{S:convergence} contain the proof of the convergence of semi-wavefronts to wavefronts while in Section \ref{s:gposneg} we prove the result corresponding to (${\rm g}_1$). 


\section{Main results}\label{s:main}
\setcounter{equation}{0}
In this section we first introduce traveling-wave and semi-wavefront solutions to \eqref{e:E}; assumptions (D) and (g) are not required in these definitions. We refer to \cite{Bonheure-Sanchez, Corli-Malaguti, GK, SanchezGarduno-Maini} for more information. Then, we state and comment our main results.

\begin{definition}\label{d:tws} Let $I\subseteq \R$ be an open interval; consider a function $\varphi \colon I \to [0, \overline{\rho}]$ such that $\varphi\in C(I)$ and $D(\varphi) \varphi^{\, \prime}\in L_{\rm loc}^1(I)$. For all $(x,t)$ with $x-ct \in I$, the function $\rho(x,t)=\varphi(x-ct)$ is said a {\em traveling-wave} solution of equation \eqref{e:E} with wave speed $c$ and wave profile $\phi$ if
\begin{equation}\label{e:def-tw}
\int_I \left\{\left(D\left(\phi(\xi)\right)\phi'(\xi) - f\left(\phi(\xi)\right) + c\phi(\xi) \right)\psi'(\xi) - g\left(\phi(\xi)\right)\psi(\xi)\right\}\,d\xi =0,
\end{equation}
for every $\psi\in C_0^\infty(I)$. A traveling-wave solution is
\begin{itemize}

\item[{\rm -)}] {\em global} if $I=\R$;

\item[{\rm -)}] {\em strict} if $I\ne \R$ and $\phi$ is not extendible to $\R$;

\item[{\rm -)}] {\em classical} if $\varphi$ is differentiable, $D(\varphi) \varphi'$ is absolutely continuous and \eqref{e:tws} holds a.e.

\item[{\rm -)}] {\em sharp at $\ell$} if $g(\ell)=0$ and there exists $\xi_0\in I$ such that $\phi(\xi_0)=\ell$, with $\phi$ classical in $I\setminus\{\xi_0\}$ and not differentiable at $\xi_0$.

\end{itemize}
A {\em wavefront solution} is a global traveling-wave solution such that the limits of $\phi$ at $\pm\infty$ are zeros of the function $g$.
\end{definition}

We point out that a profile $\phi$ to a traveling-wave solution must be differentiable {\em a.e. in $I$}; if $\phi$ is classical, then it is differentiable {\em everywhere in $I$}. In the latter case it can happen that $\phi$ extends continuously to $\bar I$ but it is not differentiable at the extreme points of $\bar I$. Of course, if  \eqref{e:tws} holds a.e. in $I$ then \eqref{e:def-tw} is satisfied. We remark that the study of {\em global} sharp traveling-wave solutions for \eqref{e:E} has been done in \cite{MMconv} under the further requirement \[
\lim_{\xi \to \xi_0}D\left(\phi(\xi)\right)\phi^{\prime}(\xi)=0.
\]
Indeed, in Proposition \ref{l:sharp} we prove that this property is a consequence of Definition \ref{d:tws}; then, the two definitions are equivalent for sharp global traveling-wave solutions. We point out that in this paper we shall only deal with classical or sharp traveling-wave solutions.

Now, we define semi-wavefront solutions.

\begin{definition}\label{d:swf}
Let $\rho$ be a traveling-wave solution of equation \eqref{e:E} whose wave profile $\phi$ is defined in $(\varpi, +\infty)$, $\varpi \in \R$; let $\ell^+\in[0,\overline{\rho}]$ be such that $g(\ell^+)=0$. Then, $\rho$ is said a \emph{semi-wavefront solution} of \eqref{e:E} {\em to} $\ell^+$ if $\phi$ is monotonic, non-constant and
$$
\varphi(\xi) \to \ell^+ \quad \text{as } \xi \to +\infty.
$$
Analogously, $\rho$ is said a \emph{semi-wavefront solution} of \eqref{e:E} \emph{from} $\ell^-$, for some $\ell^-\in[0,\overline{\rho}]$, if $g(\ell^-)=0$,  $\phi$ is defined $(-\infty, \varpi)$, is monotonic, non-constant and
$\varphi(\xi) \to \ell^-$ as $\xi \to -\infty$.
\end{definition}

Above, {\em monotonic} is meant in the weak sense: if $\xi<\xi_2$ then either $\phi(\xi_1)\le \phi(\xi_2)$ or $\phi(\xi_1)\ge \phi(\xi_2)$; analogously, {\em non-constant} stands for non-identically constant. For simplicity, in the following we use the terminology introduced for solutions to \eqref{e:E} also for profiles of such solutions. We refer to Figure \ref{f:SWFs} for a representation of some semi-wavefront profiles.

\begin{figure}[htbp]
\begin{picture}(100,100)(-80,-10)
\setlength{\unitlength}{1pt}

\put(0,0){
\put(-20,0){
\put(0,0){\vector(1,0){350}}
\put(0,0){\line(-1,0){40}}
\put(0,60){\line(1,0){350}}
\put(0,60){\line(-1,0){40}}
\put(350,8){\makebox(0,0){$\xi$}}
\put(130,-10){\vector(0,1){100}}
\put(137,87){\makebox(0,0){$\phi$}}
\put(137,67){\makebox(0,0){$\overline{\rho}$}}

\put(-40,0){
\put(0,0){\thicklines{\qbezier(80,60)(125,40)(150,0)}}
\put(80,60){\thicklines{\line(-1,0){80}}}
\put(101,42){\makebox(0,0)[br]{$\phi_2$}}
\put(30,62){\makebox(0,0)[b]{$\phi_2$}}
\put(150,-2){\makebox(0,0)[t]{$\varpi_2$}}
\put(121,34.5){\thicklines{\vector(3,-2){3}}}
\multiput(80,0)(0,5){12}{$.$}
\put(80,-2){\makebox(0,0)[tl]{$\oxi_2$}}
}

\put(40,0){
\put(0,0){\thicklines{\qbezier(80,60)(80,15)(150,0)}}
\put(80,60){\thicklines{\line(-1,0){80}}}
\put(120,17){\makebox(0,0)[br]{$\phi_3$}}
\put(60,62){\makebox(0,0)[b]{$\phi_3$}}
\put(150,-2){\makebox(0,0)[t]{$\varpi_3$}}
\put(100,20){\thicklines{\vector(3,-2){3}}}
\multiput(78,0)(0,5){12}{$.$}
\put(80,-2){\makebox(0,0)[tl]{$\oxi_3$}}
}

\put(-40,0){
\put(0,0){\thicklines{\qbezier(0,56)(70,56)(110,0)}}
\put(60,32){\makebox(0,0)[b]{$\phi_1$}}
\put(110,-2){\makebox(0,0)[t]{$\varpi_1$}}
\put(83,29){\thicklines{\vector(1,-1){3}}}
}

\put(90,0){
\put(0,0){\thicklines{\qbezier(160,0)(160,58)(260,58)}}
\put(220,43){\makebox(0,0)[b]{$\phi_5$}}
\put(164,-2){\makebox(0,0)[t]{$\varpi_5$}}
\put(191,46){\thicklines{\vector(2,1){3}}}
}

\put(50,0){
\put(0,0){\thicklines{\qbezier(160,0)(160,60)(220,60)}}
\put(192,45){\makebox(0,0)[bl]{$\phi_4$}}
\put(240,62){\makebox(0,0)[b]{$\phi_4$}}
\put(164,-2){\makebox(0,0)[t]{$\varpi_4$}}
\put(173,42){\thicklines{\vector(1,1){3}}}
\put(220,60){\thicklines{\line(1,0){80}}}

\multiput(220,0)(0,5){12}{$.$}
\put(220,-2){\makebox(0,0)[tl]{$\oxi_4$}}
}

}
}\end{picture}
\caption{\label{f:SWFs}{A strictly decreasing semi-wavefront profile $\phi_1$ from $\overline{\rho}$; a strictly increasing semi-wavefront profile $\phi_5$ to $\overline{\rho}$. Non-strictly decreasing, sharp (at $\orho$) semi-wavefront profiles $\phi_2$ and $\phi_3$ from $\overline{\rho}$; a non-strictly increasing, classical semi-wavefront profile $\phi_4$ to $\overline{\rho}$. While $\phi_4$ is smooth at $\oxi_4$, $\phi_2$ and $\phi_3$ are not smooth at $\oxi_2$ and $\oxi_3$, respectively.}}
\end{figure}
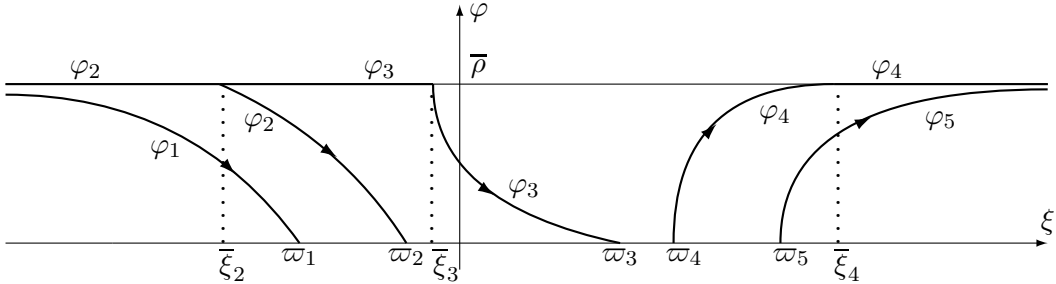



Notice that equation \eqref{e:E} with conditions (D) and (g) can only admit semi-wavefront solutions from (to) $\orho$. For a  profile $\phi$ of a semi-wavefront solution from $\orho$, we use the notation
\begin{equation}\label{e:xi0-sharp}
\oxi=\inf \left\{\xi <\varpi \, : \, \varphi(\xi)<\orho\right\}.
\end{equation}
We define $\oxi=\sup \left\{\xi >\varpi \, : \, \varphi(\xi)<\orho\right\}$ in the case of a semi-wavefront solution to $\orho$ . For every {\em sharp} semi-wavefront profile the corresponding  value $\oxi$ is a real number and coincides with $\xi_0$ introduced in Definition \ref{d:tws}, which as a consequence is unique. However, there are also classical semi-wavefronts for which $\oxi \in \mathbb{R}$ (see Theorem \ref{t:strictly}). More precisely, if a profile $\varphi$ from $\orho$ is sharp then $\oxi\in \R$ and, if $\varphi^{\, \prime}(\oxi^+)$ exists, then $\varphi^{\, \prime}(\oxi^+)\ne 0$, being possibly infinite; therefore, $\phi$ is not strictly monotone. If $\phi$ is a classical profile from $\orho$, then either $\oxi=-\infty$ or $\oxi\in\R$ and $\phi'(\oxi) =0$.

Here follows our first main result. It states that semi-wavefront solutions from (to) $\orho$ exist for every wave speed $c$; moreover, it establishes whether they are either classical or sharp, according to the different values of $c$. Roughly speaking, in the case of semi-wavefront solution from $\overline \rho$, slow profiles are sharp and fast profiles are classical while the converse holds for semi-wavefront to $\orho$. We denote by a dot the differentiation with respect to $\rho$.

\begin{theorem}[Existence of semi-wavefront solutions]\label{t:semi} Consider equation \eqref{e:E} under assumptions {\rm (D)}, or $({\rm \tilde D})$, and {\rm (g)}. Then, for every wave speed $c \in \mathbb{R}$, equation \eqref{e:E} has semi-wavefront solutions from $\overline \rho$ and to $\orho$, which are strict and unique (up to shifts) in the class of classical and sharp traveling-wave solutions. 

Moreover, let $\phi$ denote the wave profile. In case {\em (D)} we have that
\[
\text{$\phi$ is }
\left\{
\begin{array}{ll}
\text{sharp } & \hbox{ if } c<h(\orho),
\\
\text{classical } & \hbox { if } c> h(\orho),
\end{array}
\right.
\quad\hbox{ from }\orho,
\]
while
\[
\text{$\phi$ is }
\left\{
\begin{array}{ll}
\text{classical} & \hbox{ if } c<h(\orho),
\\
\text{sharp } & \hbox { if } c> h(\orho),
\end{array}
\right.
\quad\hbox{ to }\orho.
\]
In the case $c=h(\orho)$, the profiles are classical if $\dot D(\orho)<0$ while they can be either classical or sharp if $\dot D(\orho)=0$.

In case $({\rm \tilde D})$ the profiles are always classical.
\end{theorem}

We shall see in the proof, see also Remark \ref{rem:vanishing-order}, that, when (D) holds, in the case $c=h(\orho)$ and $\dot D(\orho)=0$ the possibility for a profile of being classical or sharp depends on the order of vanishing of $D$, $h-c$ and $g$ at $\orho$. Notice that the effect of assumption $({\rm \tilde D})$ consists in regularizing the profiles, in the sense that all of them are always classical. 

We now briefly consider source terms $g$ satisfying condition $(\hat{\rm g})$.
Under conditions ($\hat {\rm D}$) and ($\hat {\rm g}$),  the stationary solution of \eqref{e:E} is $u\equiv 0$; therefore, the asymptotic state of the possible semi-wavefront profiles is $0$. For simplicity, we state the following result only in the case of semi-wavefront solutions from $0$.

\begin{theorem}[A negative source term vanishing at $0$]\label{t:meno} Consider equation \eqref{e:E} under assumptions  $\hat {\rm (D)}$ and {\rm (}$\hat {\rm g}{\rm )}$. Then, for every wave speed $c \in \mathbb{R}$, equation \eqref{e:E} has a semi-wavefront solution from $0$ which is unique (up to shifts) in the class of classical or sharp solutions. Moreover, such a solution is strict.

In the case $D(0)>0$ the wave profile $\varphi$ is classical. If $D(0)=0$ then
\[
\text{$\phi$ is }
\left\{
\begin{array}{ll}
\text{sharp } & \hbox{ if } c<h(0),
\\
\text{classical } & \hbox { if } c> h(0);
\end{array}
\right.
\]
at last, if $c=h(0)$ then $\phi$ is classical if $\dot D(0)>0$, while in the case $\dot D(0)=0$ it can be either classical or sharp.
\end{theorem}

Now, we come back to assumption (D) and (g). {\em Sharp} semi-wavefront profiles cannot be strictly monotone because they are constant for $\xi\in(-\infty,\oxi]$ or for $\xi\in [\oxi, \infty)$. On the contrary, {\em classical} profiles can be either strictly or non-strictly monotone; the following result gives some simple conditions on the forcing term $g$ that show when this happens. An analogous result was given in \cite{Corli-Malaguti} by exploiting the assumption $D(\orho)>0$; however, its proof does not extend straightforwardly to cover the case when (D) holds.

\begin{theorem}[Characterization of strictly monotone solutions]\label{t:strictly} Consider equation \eqref{e:E} under assumptions  {\rm (D)} and {\rm (g)}. Let $\phi$ be a wave profile with wave speed $c$ of some semi-wavefront solution from (to) $\orho$ and let $L>0$ be a constant.
\begin{itemize}
\item[{(i)}] If
$g(\rho)\le L(\overline \rho-\rho)$ in a left neighborhood of $\orho$
and 
$c>h(\overline\rho)$ (resp., $c<h(\orho)$), then $\phi$ is strictly monotone, i.e., $\varphi(\xi)<\overline \rho$ for every $\xi$ in its domain.

\item[{(ii)}] If
$g(\rho)\ge L(\overline \rho-\rho)^{\alpha}$ in a left neighborhood of $\orho$ 
for some $\alpha \in (0,1)$, then $\phi$ is non-strictly monotone, i.e., $\varphi(\xi)\equiv \overline \rho$ in $(-\infty, \overline \xi]$ (resp., in $[\overline \xi,+\infty)$), for some $\overline \xi$ in its domain.
\end{itemize}
\end{theorem}

The borderline case $c=h(\orho)$ is not considered in case {\em (i)} above, since it involves a heavier technical analysis. We refer to Figure \ref{f:slow-fast} for a graphical representation of Theorem \ref{t:semi} and Theorem \ref{t:strictly}. The extension of Theorem \ref{t:strictly} to the case when assumption ($\hat {\rm g}$) holds is straightforward and, then, omitted.


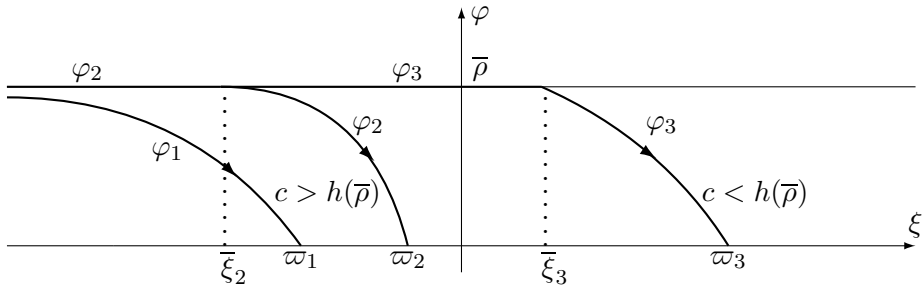
\begin{figure}[htbp]
\begin{picture}(100,100)(-80,-10)
\setlength{\unitlength}{1pt}

\put(0,0){
\put(0,0){\vector(1,0){300}}
\put(0,0){\line(-1,0){40}}
\put(0,60){\line(1,0){300}}
\put(0,60){\line(-1,0){40}}
\put(300,8){\makebox(0,0){$\xi$}}
\put(130,-10){\vector(0,1){100}}
\put(137,87){\makebox(0,0){$\phi$}}
\put(137,67){\makebox(0,0){$\overline{\rho}$}}

\put(80,25){\makebox(0,0)[t]{$c>h(\orho)$}}

\put(-40,0){
\put(0,0){\thicklines{\qbezier(80,60)(135,60)(150,0)}}
\put(80,60){\thicklines{\line(-1,0){80}}}
\put(141,42){\makebox(0,0)[br]{$\phi_2$}}
\put(30,62){\makebox(0,0)[b]{$\phi_2$}}
\put(150,-2){\makebox(0,0)[t]{$\varpi_2$}}
\put(134.5,35){\thicklines{\vector(1,-1){3}}}
\multiput(80,0)(0,5){12}{$.$}
\put(80,-2){\makebox(0,0)[tl]{$\oxi_2$}}
}

\put(-40,0){
\put(0,0){\thicklines{\qbezier(0,56)(70,56)(110,0)}}
\put(60,32){\makebox(0,0)[b]{$\phi_1$}}
\put(110,-2){\makebox(0,0)[t]{$\varpi_1$}}
\put(83,29){\thicklines{\vector(1,-1){3}}}
}

\put(80,0){
\put(0,0){\thicklines{\qbezier(80,60)(125,40)(150,0)}}
\put(80,60){\thicklines{\line(-1,0){120}}}
\put(131,42){\makebox(0,0)[br]{$\phi_3$}}
\put(30,62){\makebox(0,0)[b]{$\phi_3$}}
\put(150,-2){\makebox(0,0)[t]{$\varpi_3$}}
\put(120,35){\thicklines{\vector(3,-2){3}}}
\multiput(80,0)(0,5){12}{$.$}
\put(80,-2){\makebox(0,0)[tl]{$\oxi_3$}}
\put(160,25){\makebox(0,0)[t]{$c<h(\orho)$}}
}

}
\end{picture}
\caption{\label{f:slow-fast}{A strictly decreasing classical profile $\phi_1$ occurring in case {\em (i)}; a non-strictly decreasing classical profile $\phi_2$ occurring in case {\em (ii)}; a sharp profile $\phi_3$.}}
\end{figure}

We give now a result of convergence of semi-wavefronts to wavefronts. As we mentioned in the Introduction, if $g_0$ satisfies (${\rm g}_0$) then the corresponding equation
\begin{equation}\label{e:e0}
\rho_t + f(\rho)_x=\left(D(\rho)\rho_x\right)_x+g_0(\rho), \qquad t\ge 0, \, x\in \R,
\end{equation}
has a wavefront solution connecting $\orho$ with $0$, for every wave speed  $c\ge c_0^*$; the corresponding profile is decreasing and estimates are available for the threshold speed $c_0^*$, see \cite{Bonheure-Sanchez, GK, Malaguti-Marcelli_2002}. We also consider a strictly decreasing sequence $\{g_n\}_{n\ge1}$ of source terms satisfying condition (g) and converging uniformly to $g_0$, see Figure \ref{f:gs}. As proved in Theorem \ref{t:semi}, the corresponding equations
\begin{equation}\label{e:en}
\rho_t + f(\rho)_x=\left(D(\rho)\rho_x\right)_x+g_n(\rho), \qquad t\ge 0, \, x\in \R,\quad n\ge1,
\end{equation}
admit semi-wavefront solutions from and to $\orho$, for every wave speed $c$. For  simplicity, we restrict our discussion to the significative cases $({\rm D})$ and $({\rm \tilde D})$.


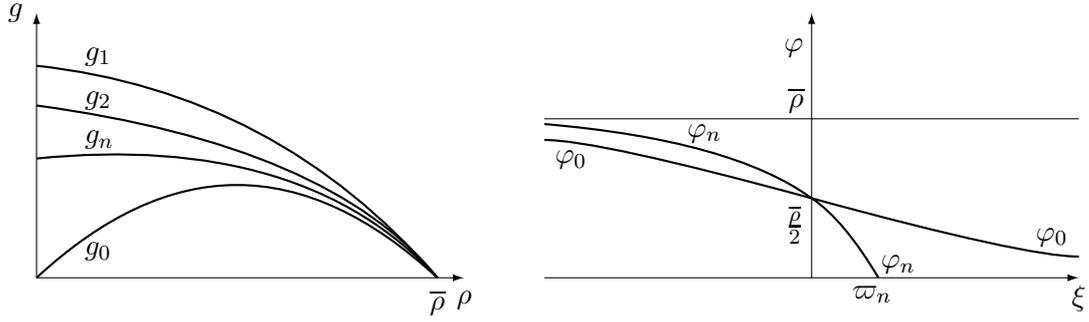
\begin{figure}[htbp]
\begin{picture}(100,110)(-80,-10)
\setlength{\unitlength}{1pt}

\put(0,0){
\put(-60,0){
\put(0,0){\vector(1,0){160}}
\put(160,-5){\makebox(0,0)[t]{$\rho$}}
\put(0,0){\vector(0,1){100}}
\put(-5,100){\makebox(0,0)[r]{$g$}}
\put(150,-5){\makebox(0,0)[t]{$\overline{\rho}$}}

\put(0,0){\thicklines{\qbezier(0,0)(75,70)(150,0)}} 
\put(18,10){\makebox(0,0)[l]{$g_0$}}
\put(0,0){\thicklines{\qbezier(0,80)(90,70)(150,0)}} 
\put(18,80){\makebox(0,0)[lb]{$g_1$}}
\put(0,0){\thicklines{\qbezier(0,65)(110,50)(150,0)}} 
\put(18,67){\makebox(0,0)[l]{$g_2$}}
\put(0,0){\thicklines{\qbezier(0,45)(100,55)(150,0)}} 
\put(18,52){\makebox(0,0)[l]{$g_n$}}
}

\put(230,0){
\put(0,0){\vector(1,0){100}}
\put(0,0){\line(-1,0){100}}
\put(0,60){\line(1,0){100}}
\put(0,60){\line(-1,0){100}}
\put(100,-3){\makebox(0,0)[t]{$\xi$}}
\put(0,0){\vector(0,1){100}}
\put(-3,87){\makebox(0,0)[r]{$\phi$}}
\put(-3,67){\makebox(0,0)[r]{$\overline{\rho}$}}
\put(-3,27){\makebox(0,0)[rt]{$\frac{\overline{\rho}}{2}$}}

\put(0,0){\thicklines{\qbezier(-100,52)(-80,52)(0,30)}}
\put(0,0){\thicklines{\qbezier(0,30)(80,8)(100,8)}}
\put(-90,48){\makebox(0,0)[t]{$\phi_0$}}
\put(90,12){\makebox(0,0)[b]{$\phi_0$}}

\put(0,0){\thicklines{\qbezier(-100,58)(-30,52)(0,30)}}
\put(0,0){\thicklines{\qbezier(0,30)(12,22)(25,0)}}
\put(-40,58){\makebox(0,0)[t]{$\phi_n$}}
\put(26,5){\makebox(0,0)[l]{$\phi_n$}}
\put(30,-3){\makebox(0,0)[tr]{$\varpi_n$}}

}
}
\end{picture}
\caption{\label{f:gs}{Left: the function $g_0$ and the sequence $\{g_n\}_{n\ge1}$. Right: the profiles $\phi_0$ and $\phi_n$.}}
\end{figure}

Since profiles $\phi_n$, $n\ge0$, to either \eqref{e:e0} or \eqref{e:en} are uniquely defined only up to shifts, we fix their values at $\xi=0$ by imposing
\begin{equation}\label{e:phi(0)}
\phi_n(0) = \frac{\orho}{2},\quad n\ge0.
\end{equation}
We can now state our convergence result, see Figure \ref{f:gs}.

\begin{theorem}[Convergence of semi-wavefront profiles to a wavefront profile]\label{t:convergence} Assume either condition $({\rm \hat D})$ or $({\rm \tilde D})$. Consider $g_0\in C[0,\orho]$ satisfying $({\rm g}_0)$ and let $\{g_n\}_{n\ge1}$ be a decreasing sequence satisfying \emph{(g)} such that $g_n\to g_0$ uniformly. Moreover, for $c \ge c_0^*$  let $\phi_0$ be the wavefront profile of \eqref{e:e0} and $\phi_n$ the semi-wavefront profile of \eqref{e:en} from $\orho$, both of them with wave speed $c$ and satisfying \eqref{e:phi(0)}.

Then $\phi_n\to\phi_0$ in $C^1_{\rm loc}(J)$, where $J$ is the maximal open interval where $0<\phi_0 <\orho$.
\end{theorem}

Theorem \ref{t:convergence} deserves some comments. First, note that $\phi_0$ can be either strictly monotone or not; in the latter case, it can be either classic or sharp at one or even both equilibria of $g$. Second, an analogous result holds for semi-wavefront profiles to $\orho$. Third, much more general results can be given if we also let the diffusivity, flux and wave speed vary and converge to some limit functions \cite{MMM10}. We focused on the source terms because they determine whether solutions are either semi-wavefronts (in case (g) holds) or wavefronts (in case $({\rm g}_0)$).

Now, we assume that the source term $g\in C[0,\orho]$ satisfies (${\rm g}_1$) and
\begin{equation}\label{eq:goodg}
|g(\rho)|\ge L\left|\rho_0-\rho\right|^{\alpha}\quad \hbox{in a neighborhood of $\rho_0$,}
\end{equation}
for some $\alpha \in (0,1)$ and $L>0$; see Figure \ref{f:gchanges}. We aim at constructing traveling-wave solutions whose profiles are defined {\em through} the equilibrium point $\rho_0$; this will be obtained through a suitable pasting of some semi-wavefront solutions. 


\begin{figure}[htbp]
\begin{picture}(100,120)(-80,-10)
\setlength{\unitlength}{1pt}

\put(60,40){
\put(0,0){\vector(1,0){200}}
\put(200,-5){\makebox(0,0)[t]{$\rho$}}
\put(0,0){\vector(0,1){75}}
\put(0,0){\line(0,-1){50}}
\put(-5,75){\makebox(0,0)[r]{$g$}}
\multiput(149,-0.5)(0,-5){12}{$.$}
\put(150,5){\makebox(0,0)[b]{$\overline{\rho}$}}
\put(78,5){\makebox(0,0)[bl]{$\rho_0$}}

\put(0,0){\thicklines{\qbezier(0,60)(35,50)(75,0)}} 
\put(0,0){\thicklines{\qbezier(75,0)(100,-30)(150,-50)}} 
\put(20,45){\makebox(0,0)[t]{$g_1$}}

\put(0,0){\thicklines{\qbezier(0,65)(75,50)(75,0)}} 
\put(0,0){\thicklines{\qbezier(75,0)(75,-30)(150,-58)}} 
\put(20,70){\makebox(0,0)[t]{$g_2$}}

}
\end{picture}
\caption{\label{f:gchanges}{Two functions satisfying (g${}_1$): here, $g_2$ satisfies \eqref{eq:goodg} while $g_1$ does not.}}
\end{figure}
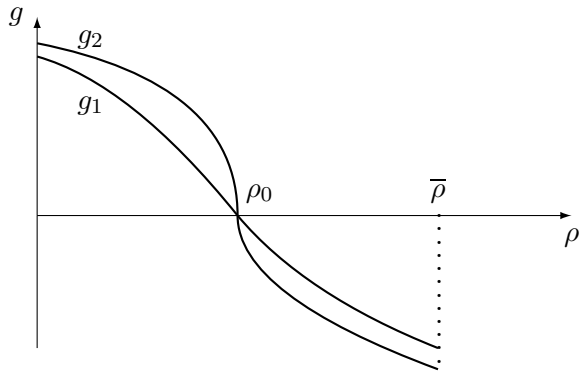

\begin{theorem}[Existence of traveling-wave solutions]\label{t:tws} Consider equation \eqref{e:E} under assumptions {\rm ($\hat {\rm D}$)}, {\rm(${\rm g}_1$)} and \eqref{eq:goodg}. Then, for every wave speed $c \in \mathbb{R}$ equation \eqref{e:E} has for solutions:
\begin{enumerate}[(1)]

 \item a traveling wave $\phi_1$ assuming any value in $[0, \orho]$ and with a strictly increasing profile;

 \item a traveling wave $\phi_2$ assuming any value in $[0, \orho]$ and with a strictly decreasing profile;

 \item a traveling wave $\phi_3$ with values in the interval $[0, \rho_0]$;

 \item a traveling wave $\phi_4$ with values in the interval $[\rho_0,\orho]$.

 \end{enumerate}
All these traveling-wave solutions are strict and classical; moreover, they are unique (up to shifts) in the class of classical and sharp strictly monotone traveling-wave solutions.
\end{theorem}
 
We refer to Section \ref{s:gposneg} for a pictorial interpretation of this result. 


\section{The first-order problem}\label{s:first solution}
\setcounter{equation}{0}

For brevity, in the following we simply refer to $\rho$ as a semi-wavefront and to $\phi$ as its profile. Moreover, we mainly consider the case of semi-wavefronts {\em from} $\orho$; the case of semi-wavefront solutions {\em to} $\orho$ is analogous.

This section is devoted to the singular first-order boundary value problem
\begin{equation}\label{e:fo}
\left\{
\begin{array}{l}
\dot z(\varphi) = h(\varphi)-c-\frac{D(\varphi)g(\varphi)}{z(\varphi)}, \\
z(\varphi)<0,  \quad \varphi \in (0,\overline \rho),\\
z(0\,^+)=:z_0 \le 0, \quad z(\overline{\rho}\,^-)=0.
\end{array}
\right.
\end{equation}
We used the notation $z(0^+)$ and $z(\orho^-)$ because the equation in $\eqref{e:fo}$ is singular: its right-hand side is not defined at $\orho$ and possibly it is defined neither at $0$; then the values of $z$ at these points must be understood in the sense of the limit. As a consequence, solutions $z$ to \eqref{e:fo} are meant in the sense $z\in C^0[0,\overline{\rho}]\cap C^1(0,\overline{\rho})$. We point out that the differentiability of $D$ plays no role in the solvability of \eqref{e:fo}.

\begin{lemma}\label{l:Ex-sing2017}
Assume  {\rm (g)} and let $D \in C[0, \orho]$ be such that $D(\rho)>0, \, \rho \in (0, \orho)$. Assume one of the following conditions:
\begin{enumerate}[(i)]
\item $D(0)>0$;

\item $D(0)=0$ and $\displaystyle{\limsup_{\phi \to 0^+}} \frac{D(\phi)}{\phi}<\infty$;

\item $D(0)=0$ and $\dot D(0)=\infty$.
\end{enumerate}
Then, problem \eqref{e:fo} is uniquely solvable for every $c\in\R$. Moreover,
there exists a real number $c^*$ such that
\[
z(0^+)= \left\{
\begin{array}{ll}
0 & \hbox{ in case (ii)  when } c\ge c^*,
\\
z_0<0 & \hbox{ otherwise}.
\end{array}
\right.
\]
\end{lemma}

\begin{proof} The proof of this result already appeared in previous papers: case \emph{(i)} and case \emph{(ii)} with $c<c^*$ are treated in \cite[Theorem 2.6]{Corli-Malaguti}, case  \emph{(iii)} is discussed in \cite[Theorem 2.10]{Corli-Malaguti}, case \emph{(ii)} with $c\ge c^*$ can be obtained by \cite[Lemma 2.2 and Theorem 4.1]{Malaguti-Marcelli_2002} when assuming in \cite{Malaguti-Marcelli_2002} $d=1$ and the source term equal to $Dg$. Notice that in \cite{Malaguti-Marcelli_2002} $D(0)>0$ and $g(0)=0$;  here the assumptions on $D$ and $g$ in $\rho=0$ are exchanged, but this does not affect the conclusion.
\end{proof}

Now, we show that the solution $z$ provided by Lemma \ref{l:Ex-sing2017} is differentiable at $\orho$ and that $\dot z(\orho)$ can be explicitly computed. Notice that by \eqref{e:zeq} we have
\begin{equation}\label{e:zmod}
\left(\dot z(\varphi) -h(\varphi) +c  \right)\frac{z(\varphi)}{\varphi-\orho} =-\frac{D(\varphi)}{\varphi-\orho}\,g(\varphi).
\end{equation}
Hence, it is clear that the value of $\dot z(\orho)$ depends on the behavior of the right-hand side in \eqref{e:zmod} near the point $\orho$. This accounts for the following statement.

\begin{proposition}\label{p:slopePDnew} Under the same hypotheses of Lemma \ref{l:Ex-sing2017}, assume moreover that the limit of the right-hand side of \eqref{e:zmod} exists and denote
\[
\lim_{\phi\to\orho^-}\frac{D(\varphi)}{\varphi-\orho}\,g(\varphi) = \ell\in(-\infty,0].
\]
Then, the solution $z$ of problem \eqref{e:fo} satisfies
%
\begin{equation}\label{e:z-slope}
\dot z(\orho) = \left\{
\begin{array}{ll}
\left\{
\begin{array}{ll}
0 & \hbox{ if } c\ge h(\orho),
\\
h(\orho)-c & \hbox{ if } c<h(\orho),
\end{array}
\right.
& \hbox{ if } \ell=0,
\\[2mm]
\ds\frac{h(\orho)-c+\sqrt{(h(\orho)-c)^2-4\ell}}{2} & \hbox{ if } -\infty<\ell<0,
\end{array}
\right.
\end{equation}

\end{proposition}
\begin{proof}
We consider separately two cases.

\smallskip
{\em (i)} $\ell=0$. This is the case, in particular, when (D) is satisfied; most of the following proof already appeared in \cite[Lemma 2.1]{MMconv}. We denote the lower and upper left Dini-derivatives of $z$ at $\orho$ by
    \begin{equation*}
  D_-z(\orho)=:\liminf_{\phi \to \orho\,^-}\frac{z(\phi)}{\phi-\orho}, \qquad \limsup_{\phi \to \orho\,^-}\frac{z(\phi)}{\phi-\orho}:=D^-z(\orho).
    \end{equation*}
By $\eqref{e:fo}_2$ we have $D_-z(\orho)\ge 0$.

If $D_-z(\orho)>0$, then $\frac{z(\phi)}{\phi-\orho} \ge \delta>0$ in some left neighborhood of $\orho$. By \eqref{e:zmod} and $\ell=0$ we deduce
\begin{equation*}
\lim_{\phi \to \orho\,^-}\dot z(\phi)=
h(\orho)-c,
\end{equation*}
and this leads to the existence of $\dot z(\orho)=h(\orho)-c$.

If $D_-z(\orho)=0$, to prove that $\dot z(\orho)$ exists we argue by contradiction and then assume $D^-z(\orho)>0$. As a consequence, for every $\lambda \in
\left(0, D^-z(\orho)\right)$ we can find two sequences $\{\alpha_n\}$ and $\{\beta_n\}$ in $(0,\orho)$, both of them converging to $\orho$, such that
\begin{equation}\label{e:alphabeta}
\begin{array}{ccc}
\ds\frac{z(\alpha_n)}{\alpha_n-\orho}=\lambda &\mbox{ and } &
\ds\frac{\dot z(\alpha_n) -
\lambda}{\alpha_n-\orho} = \frac{d}{d\varphi}\left(
\frac{z(\varphi)}{\varphi -\orho} \right )_{| \varphi=\alpha_n}\ge 0;\\
\\
\ds\frac{z(\beta_n)}{\beta_n-\orho}=\lambda &\mbox{ and } &
\ds\frac{\dot z(\beta_n) -
\lambda}{\beta_n-\orho} = \frac{d}{d\varphi}\left(
\frac{z(\varphi)}{\varphi -\orho} \right )_{| \varphi=\beta_n}\le 0.
\end{array}
\end{equation}
By $\eqref{e:alphabeta}_2$ we have $\dot z(\beta_n) \ge
\frac{z(\beta_n)}{\beta_n - \orho}=\lambda$ and then
\begin{equation*}
\left(h(\beta_n) -c  -\dot z(\beta_n) \right)\frac{z(\beta_n)}{\beta_n-\orho}\le -\lambda^2-\lambda\left(c-h(\beta_n)\right).
\end{equation*}
If $c \ge h(\orho)$ we have $\lim_{n\to\infty}-\lambda^2-\lambda\left(c-h(\beta_n)\right)= -\lambda^2-\lambda(c-h(\orho))\le -\lambda^2$. This contradicts \eqref{e:zmod} because of $\ell=0$.

If $c <h(\orho)$, then we can choose $\lambda<h(\orho)-c$ and get, by \eqref{e:zmod},
$$
\dot
z(\alpha_n)=h(\alpha_n)-c-\frac{\frac{D(\alpha_n)g(\alpha_n)}{\alpha_n-\orho}}{\lambda}
\to h(\orho)-c>\lambda,
$$
when $n \to \infty$. But by $\eqref{e:alphabeta}_1$ we have $\dot z(\alpha_n)\le \lambda$, a contradiction.

Therefore, up to now we proved that $z$ is differentiable at $\orho$. The assumption $\ell=0$ together with \eqref{e:zmod} imply
\begin{equation}\label{e:either-or}
\hbox{ either\quad  $\dot z(\orho)=0$\quad or\quad $\dot z(\orho)=h(\orho) -c$.}
\end{equation}
Now, we prove $\eqref{e:z-slope}_1$. If $c=h(\orho)$, then we have $\dot z(\orho)=0$. If $c>h(\orho)$, then $\dot z(\orho)=h(\orho)-c$ should imply $z(\phi)>0$ in a left neighborhood of $\orho$, a contradiction; then, $\dot z(\orho)=0$. If $c< h(\orho)$, by the positivity of both $D$ and $g$ it follows that every solution $z$ of problem \eqref{e:fo} satisfies
$$
\dot z(\varphi)=h(\varphi)-c-\frac{D(\varphi)g(\varphi)}{z(\varphi)}>h(\varphi)-c, \qquad \varphi\in(0,\orho),
$$
and then $\ds\liminf_{\varphi \to \orho^-}\dot z(\varphi)\ge h(\orho)-c$. Since $c<h(\orho)$, we can find $\eta>0$ such that
$$
\dot z(\varphi)\ge h(\orho)-c- \frac{h(\orho)-c}{2}=\frac{h(\orho)-c}{2}=:\sigma>0, \qquad \varphi \in (\orho-\eta, \orho).
$$
As a consequence, by the Mean Value Theorem we have
$$
-z(\varphi)=-z(\varphi) +z(\orho)=\dot z(\xi)(\orho-\varphi)>\sigma (\orho - \varphi), \qquad \varphi \in (\orho-\eta, \orho),
$$
with $\xi \in (\varphi , \orho)$. Therefore we deduce $z(\varphi)<\sigma(\varphi-\orho) $ for $\varphi \in (\orho-\eta, \orho)$ and, in turn, $\dot z(\orho)\ge \sigma$. By \eqref{e:either-or} we conclude that $\dot z(\orho)=h(\orho)-c$.

\smallskip
{\em (ii)}  $-\infty<\ell<0$. We argue again by contradiction. If $\dot z(\orho)$ does not exist, then $0\le D_-z(\orho)<D^-z(\orho)\le \infty$ and we can find sequences $\{\alpha_n\}$, $\{\beta_n\}$ as in \eqref{e:alphabeta} for any $\lambda \in \left(D_-z(\orho), D^-z(\orho)\right)$. By \eqref{e:zmod} we have
    \begin{equation*}
    \lambda\ge \dot z(\alpha_n)=h(\alpha_n)-c-\frac{\frac{D(\alpha_n)g(\alpha_n)}
    {\alpha_n-\orho}
    }{\lambda}\to h(\orho)-c-\frac{\ell}{\lambda}, \qquad\text{ as }n\to\infty,
    \end{equation*}
and then $\lambda^2 -\left(h(\orho)-c\right)\lambda+\ell\ge 0$. Similarly, by means of $\{\beta_n \}$, we obtain that $\lambda^2 -(h(\orho)-c)\lambda+\ell\le 0$. The two last inequalities and the sign condition of $\lambda$ imply
    \begin{equation*}
    \lambda=\frac{h(\orho)-c+\sqrt{(h(\orho)-c)^2-4\ell}}{2}.
    \end{equation*}
This contradicts the arbitrariness of $\lambda$. Hence $\dot z(\orho)$ exists and we denote $\mu :=\dot z(\orho)\in[0,\infty]$. By the assumption $-\infty<\ell<0$, from \eqref{e:zmod} we obtain that $\dot z(\phi)$ has a limit for $\phi \to \orho\,^-$ and it is necessarily $\mu$. Again by \eqref{e:zmod}, by passing to the limit for $\phi\to\orho^-$ we deduce $\mu^2-\left(h(\orho)-c\right)\mu+\ell=0$ and, since $\mu \ge 0$, this implies $\eqref{e:z-slope}_2$.
\end{proof}

We notice that $\dot z$ is continuous at $\orho$ if $\dot z(\orho)\ne 0$; this smoothness is not granted in general.

\section{Existence of semi-wavefront solutions}\label{s:nature}
\setcounter{equation}{0}

We first prove in this section that, even if {\em sharp} profiles lose regularity at the point $\oxi$ where they reach the value $\orho$, nevertheless some smoothness still holds, see \eqref{e:Dphi'vanishing}. Indeed,
that property is also satisfied by every non-strictly monotone {\em classical} profile. In the second part of the section we prove Theorems \ref{t:semi} and \ref{t:meno}.

\begin{proposition}\label{l:sharp}
Assume {\rm (D)} and {\rm (g)}. If $\varphi$ is a classical or sharp semi-wavefront profile from $\orho$, then
\begin{equation}\label{e:Dphi'vanishing}
\lim_{\xi \to \oxi}D\left(\phi(\xi)\right)\phi^{\prime}(\xi)=0,
\end{equation}
where $\oxi$ is defined in \eqref{e:xi0-sharp}. An analogous result holds in the case of semi-wavefront profiles to $\orho$.
\end{proposition}

\begin{proof}
For simplicity, in the following we only consider the case of profiles from $\orho$. If $\varphi$ is a classical profile, then $\varphi^{\prime}(\xi) \to 0$ as $\xi \to -\infty$ \cite[Lemma 6.4]{Corli-Malaguti} and \eqref{e:Dphi'vanishing} is satisfied. If $\varphi$ is sharp, then $\overline \xi\in\R$ and
\begin{equation}\label{eq:def-phi}
\varphi(\xi) = \orho, \quad \mbox{for all } \xi \le \overline \xi.
\end{equation}
If $\varphi^{\, \prime}(\overline{\xi}^+)\in\R$, then \eqref{e:Dphi'vanishing} is satisfied again. Therefore, it remains to consider the case
\begin{equation}\label{e:danger}
\varphi^{\, \prime}(\overline{\xi}^+)=-\infty,
\end{equation}
see profile $\phi_3$ in Figure \ref{f:SWFs}. We fix $\eps>0$, denote $I_\eps=(\overline{\xi}-\eps, \overline{\xi}+\eps)\subset (-\infty,\varpi)$ and consider $\psi\in C_0^{\infty}(I_\eps)$. It follows from \eqref{e:def-tw} that
\begin{align}
0=&\int_{I_\eps} \left\{\left(D(\phi)\phi' - f(\phi) + c\phi \right)\psi' - g(\phi)\psi\right\}\,d\xi\notag
\\
=&
\left(\int_{\overline{\xi}-\eps}^{\overline{\xi}} + \int_{\overline{\xi}}^{\overline{\xi}+\eps} \right)\left\{\left(D(\phi)\phi' - f(\phi) + c\phi \right)\psi' - g(\phi)\psi\right\}\,d\xi.\label{eq:I-2}
\end{align}
About the first integral in \eqref{eq:I-2}, from {\rm (g)} and \eqref{eq:def-phi} we deduce
\begin{align}
&\int_{\overline{\xi}-\eps}^{\overline{\xi}}  \left\{\left(D(\phi)\phi' - f(\phi) + c\phi \right)\psi' - g(\phi)\psi\right\}
\nonumber
\\
=&
\left[c\orho - f(\orho) \right]\int_{\overline{\xi}-\eps}^{\overline{\xi}}\psi'\,d\xi= \left[c\orho - f(\orho) \right]\psi(\overline{\xi}).
\label{eq:I-3}
\end{align}
About the second one, we must be more careful because of \eqref{e:danger}. Then we fix $0<\delta <\eps$ and notice that
\begin{align*}
&\int_{\overline{\xi}+\delta}^{\overline{\xi}+\eps} \left\{\left(D(\phi)\phi' - f(\phi) + c\phi \right)\psi' - g(\phi)\psi\right\}\,d\xi
\\
&\qquad= -\left(D\left(\phi(\overline{\xi}+\delta)\right)\phi'(\overline{\xi}+\delta) - f\left(\phi(\overline {\xi}+\delta)\right) + c\phi(\overline{\xi}+\delta)\right)\psi(\overline{\xi}+\delta)
\\
&\qquad\quad-\int_{\overline{\xi}+\delta}^{\overline{\xi}+\eps}\left(\left(D(\varphi)\varphi^{\, \prime}\right)^{\, \prime} + \left(c-h(\varphi)\right)\varphi^{\, \prime}+g(\varphi)\right)\psi\,d\xi
\\
&\qquad= -\left(D\left(\phi(\overline{\xi}+\delta)\right)\phi'(\overline{\xi}+\delta) - f\left(\phi(\overline {\xi}+\delta)\right) + c\phi(\overline{\xi}+\delta)\right)\psi(\overline{\xi}+\delta),
\end{align*}
because \eqref{e:tws} holds a.e. in $(\overline{\xi}+\delta,\overline{\xi}+\eps)$. Then, we have
\begin{align}
&\int_{\overline{\xi}}^{\overline{\xi}+\eps}  \left\{\left(D(\phi)\phi' - f(\phi) + c\phi \right)\psi' - g(\phi)\psi\right\}\,d\xi\notag
\\
&=\lim_{\delta\to0^+}\int_{\overline{\xi}+\delta}^{\overline{\xi}+\eps} \left\{\left(D(\phi)\phi' - f(\phi) + c\phi \right)\psi' - g(\phi)\psi\right\}\,d\xi\notag
\\
&= -\lim_{\delta\to 0^+} D\left(\phi(\overline{\xi}+\delta)\right)\phi^{\prime }(\overline{\xi}+\delta)\psi(\overline{\xi}+\delta) + \left[f(\orho) - c\orho\right]\psi(\overline{\xi}).
\label{eq:I-10}
\end{align}
By combining \eqref{eq:I-2}, \eqref{eq:I-3} and \eqref{eq:I-10} we obtain
\begin{equation*}
\lim_{\delta\to 0^+} D\left(\phi(\overline{\xi}+\delta)\right)\phi^{\prime}
(\overline{\xi}+\delta)\psi(\overline{\xi}+\delta)=0.
\end{equation*}
Since we can choose $\psi$ such that $\psi(\overline{\xi})\ne0$, then 
$
D\left(\phi(\xi)\right)\phi'(\xi)\to0$ as $\xi\to\overline{\xi}^+$.
\end{proof}

Now we prove a sort of converse of Proposition \ref{l:sharp}, namely that condition \eqref{e:Dphi'vanishing} allows to extend profiles defined in a bounded interval to semi-wavefront profiles.

\begin{proposition}\label{l:sharp-ii}
Assume {\rm (D)} and {\rm (g)}. Let $\phi:(\alpha,\varpi)\to[0,\orho]$ be a monotone, non-constant, classical profile in $(\alpha,\varpi)$ with $\lim_{\xi\to\alpha^+}\phi(\xi) = \orho$. If moreover
\begin{equation}\label{e:D-vanishing-phi-alpha}
\lim_{\xi \to \alpha^+}D\left(\phi(\xi)\right)\phi^{\prime}(\xi)=0,
\end{equation}
then the function
\[
\tilde\phi(\xi) := \left\{
\begin{array}{ll}
\orho & \hbox{ if } \xi\in(-\infty,\alpha],
\\
\phi(\xi) & \hbox{ if } \xi\in(\alpha,\varpi),
\end{array}
\right.
\]
is  a semi-wavefront profile from $\orho$. A similar result holds for semi-wavefront solutions to $\orho$.
\end{proposition}

\begin{proof} We only have to show that $\tilde\phi$ is a solution in a neighborhood of $\alpha$ in the sense of Definition \ref{e:def-tw}. By contradiction, we assume that there exist an interval $(a,b)$ with $a<\alpha<b<\varpi$ and a function $\psi\in C_0^{\infty}(a,b)$ such that
\begin{equation}\label{e:def-tw-52}
\int_a^b \left\{\left(D(\tilde\phi)\tilde\phi' - f(\tilde\phi) + c\tilde\phi \right)\psi' - g(\tilde\phi)\psi\right\}\,d\xi \ne0.
\end{equation}
Notice that $D(\tilde\phi)\tilde\phi'\in L^1_{\rm loc}(-\infty,\varpi)$ because of \eqref{e:D-vanishing-phi-alpha}. We have $\tilde\phi(\xi)\equiv \orho$ for $\xi \in (a,\alpha)$ and $\tilde\phi$ is a classical solution in $(\alpha+\delta, b)$ for every positive $\delta$ with $\alpha +\delta <b$. It follows from {\rm (g)} and \eqref{e:tws} that
\begin{align*}
\lefteqn{\int_{a}^{b}\left\{\left(D(\tilde\phi)\tilde\phi' - f(\tilde\phi) + c\tilde\phi \right)\psi' - g(\tilde\phi)\psi\right\}\,d\xi}
\\
&=\left\{\int_{a}^{\alpha} + \int_{\alpha}^{b}\right\}\left\{\left(D(\tilde\phi)\tilde\phi' - f(\tilde\phi) + c\tilde\phi \right)\psi' - g(\tilde\phi)\psi\right\}\,d\xi
\\
&= \left(c\orho - f(\orho)\right)\psi(\alpha) +\lim_{\delta\to0^+} \int_{\alpha+\delta}^{b}\left\{\left(D(\phi)\phi' - f(\phi) + c\phi \right)\psi' - g(\phi)\psi\right\}\,d\xi
\\
&= \left(c\orho - f(\orho)\right)\psi(\alpha) \notag
\\
&\quad - \lim_{\delta\to 0^+} \left(D (\phi(\alpha+\delta))\phi'(\alpha+\delta) - f(\phi(\alpha+\delta )) + c\phi(\alpha+\delta)\right)\psi(\alpha+\delta)\notag
\\
&= \left(c\orho - f(\orho)\right)\psi(\alpha) +\left(f(\orho) - c\orho\right)\psi(\alpha)=0,
\end{align*}
which contradicts \eqref{e:def-tw-52}.
\end{proof}

Now, we prove Theorem \ref{t:semi}.

\begin{proofof}{Theorem \ref{t:semi}} We begin by assuming (D) and we first consider the case of semi-wavefronts from $\orho$. The case of semi-wavefronts to $\orho$ is deduced at the end of the proof by a change of variables.

We prove that the existence of a strict semi-wavefront from $\overline \rho$ of \eqref{e:E} with speed $c$ is equivalent to the solvability of the boundary-value problem \eqref{e:fo}; our reasoning also allows to distinguish between classical and sharp profiles. Since problem \eqref{e:fo} is always solvable, see Section \ref{s:first solution}, this proves the first statement of the theorem.

We begin by assuming that $\phi$ is a strict semi-wavefront profile from $\orho$; then, the solvability of \eqref{e:fo} follows as in \cite[Theorem 2.5]{Corli-Malaguti}. Indeed, the profile $\phi$ is invertible \cite[Proposition 6.1]{Corli-Malaguti} and its inverse function $\xi=\xi(\varphi)$ is defined for $\varphi \in [0, \orho)$ \cite[Remark 6.3]{Corli-Malaguti}. The function $z(\varphi)=D(\varphi)\varphi^{\, \prime}\left(\xi(\varphi)\right)$ satisfies the equation in $\eqref{e:fo}$; moreover, $z(\phi)<0$ for $\phi\in(0,\orho)$ and $z(0^+)\le0$ \cite[Lemma 6.1 {\em (ii)}]{Corli-Malaguti}.  Finally, we have that $z(\orho^{\, -})=0$ by the property $\lim_{\xi \to \overline{\xi}}D\left(\varphi(\xi)\right)\varphi^{\prime}(\xi)=0$ by Proposition \ref{l:sharp}. Therefore $z$ satisfies problem \eqref{e:fo}.

\smallskip

Conversely, assume that problem \eqref{e:fo} is solvable. The proof of \cite[Theorem 2.5]{Corli-Malaguti} exploits the assumption $D(\orho)>0$ and then must be suitably adapted to the current situation. Let $z(\varphi)$ be a solution of \eqref{e:fo} for some $c\in \mathbb{R} $ and $\varphi(\xi)$ be the solution of the initial-value problem
\begin{equation}
\left\{
\begin{array}{l}
\varphi^{\, \prime}(\xi)=\frac{z(\varphi)}{D(\varphi)},\\
\varphi(0)=\frac{\overline \rho}{2},
\end{array}
\right.
\label{e:phirho2}
\end{equation}
in its maximal existence interval $(\alpha, \varpi)$, for $\alpha\in[-\infty,\varpi)$; this means that $\phi$ satisfies
\begin{equation*}
\lim_{\xi \to \alpha^+} \varphi(\xi)=\overline \rho, \qquad \qquad \lim_{\xi \to \varpi^-} \varphi(\xi)=0.
\end{equation*}
The proof that $\phi$ is a strict solution, i.e. $\varpi \in \R$, is analogous to that of \cite[Theorem 2.2]{Corli-Malaguti} because it only involves the values of $D$ near $\varphi=0$. Then, it remains to investigate the behavior of $\varphi(\xi)$ near $\xi=\alpha$ and, in particular, to describe the type of $\phi$ at that point.

Let $\overline \xi$ be as in \eqref{e:xi0-sharp} and notice that, by definition of $z$,
\begin{equation}\label{e:quotient}
\lim_{\xi \to \overline \xi^+} \varphi^{\, \prime}(\xi)=\lim_{\varphi \to \orho^-}\frac{z(\varphi)}{D(\varphi)}.
\end{equation}
Whether $\phi$ is sharp or classical depends on the value of the limit in the right-hand side of \eqref{e:quotient}, which in turn depends on the value of $c$ because of Proposition \ref{p:slopePDnew}; notice that $\ell=0$ in Proposition \ref{p:slopePDnew} since we are assuming (D). We discuss these cases.

\smallskip
{\em (i)} \emph{Case $c<h(\orho)$.} By Proposition \ref{p:slopePDnew} we have
$$
\lim_{\xi \to \overline \xi^+} \varphi^{\, \prime}(\xi)=\lim_{\varphi \to \orho^-} \frac{z(\varphi)}{D(\varphi)}
= \left\{
\begin{array}{ll} \frac{h(\orho)-c}{\dot D(\orho)} &\dot D(\orho)<0,
\\
-\infty & \dot D(\orho)=0.
\end{array}
\right.
$$
In both cases this implies $\overline \xi \in \R$ and then $\varphi(\overline \xi)=\orho$. Moreover, we have
\[
\lim_{\xi \to \overline \xi^+} D\left(\varphi(\xi)\right)\varphi^{\, \prime}(\xi) = \lim_{\varphi \to \orho^-} z(\varphi)=0.
\]
Then
\begin{equation}\label{e:tilde-phi}
\tilde \varphi(\xi)=\left\{
\begin{array}{ll} \orho & \text{ if } \xi \in(-\infty, \overline \xi],
\\
\varphi(\xi) & \text{ if } \xi \in  (\overline \xi, \varpi),
\end{array}
\right.
\end{equation}
is a \emph{sharp} semi-wavefront profile from $\rho$ by Proposition \ref{l:sharp-ii}.

\smallskip
{\em (ii)} \emph{Case $c\ge h(\orho)$ and $\dot D(\orho)<0$.} Arguing as in case {\em (i)} we obtain that the limit in \eqref{e:quotient} is $0$. If $\overline \xi =-\infty$, then the function $\varphi$ is a \emph{classical} semi-wavefront profile from $\orho$. If $\overline \xi \in \R$, then the function $\tilde \varphi$ defined in \eqref{e:tilde-phi} is a \emph{classical} semi-wavefront profile from $\orho$.

\smallskip
{\em (iii)} \emph{Case $c> h(\orho)$ and $\dot D(\orho)=0$.} In this case the situation is more delicate and we need to introduce an upper-solution for the equation in \eqref{e:fo}. Let $(\psi_n)\subset (0, \orho)$ be a sequence converging to $\orho$.  By Proposition \ref{p:slopePDnew} we deduce
    $$
    \frac{z(\psi_n)}{\psi_n-\orho}\to \dot z(\orho)=0
    $$
    as $n \to \infty$. By applying the Mean Value Theorem in every interval $[\psi_n, \orho]$, we obtain a new sequence $(\varphi_n)\subset (\psi_n, \orho)$, which again converges to $\orho$ and satisfies $\dot z(\varphi_n)\to 0$ as $n \to \infty$. Therefore, by \eqref{e:zeq}, we have
    $$
    \frac{D(\varphi_n)g(\varphi_n)}{z(\varphi_n)}\to h(\orho)-c
    $$
    and then
    \begin{equation}\label{e:CTA}
    \frac{z(\varphi_n)}{D(\varphi_n)}\to 0
    \end{equation}
    because $g(\orho)=0$. Fix $\varepsilon >0$ and denote $\eta(\varphi):=-\varepsilon D(\varphi)$. By (D) and (g) we have
    \[
    h(\varphi)-c -\frac{D(\varphi)g(\varphi)}{\eta(\varphi)}=h(\varphi)-c +\frac{g(\varphi)}{\varepsilon}\to h(\orho)-c< 0 \quad \hbox{ as }\quad \varphi \to \orho^-.
    \]
    Since $\dot \eta(\orho)=0$, we can find $\delta>0$ such that
    \begin{equation}\label{e:sps}
    \dot \eta(\varphi)> h(\varphi)-c -\frac{D(\varphi)g(\varphi)}{\eta(\varphi)}, \qquad \varphi \in (\orho-\delta, \orho).
    \end{equation}
    Hence, the function $\eta$ is an upper-solution for \eqref{e:zeq} by \eqref{e:sps} in the interval $(\orho-\delta, \orho)$. For any $\varphi\in (\orho-\delta, \orho)$, by \eqref{e:CTA} there exists $\varphi_N \in( \varphi,\orho)$ satisfying $z(\varphi_N) > -\varepsilon D(\varphi_N)=\eta(\varphi_N)$. By a classical comparison argument \cite[Lemma 3.2 \emph{2(ii)}]{Corli-Malaguti}{L: \cite[Lemma 4.2 \emph{2(ii)}]{Corli-Malaguti}}, we obtain that $\eta(\sigma)<z(\sigma)$ for $\sigma\in (\orho-\delta, \varphi_N]$ and then, since $\varphi\in (\orho-\delta, \orho)$ was arbitrary,
    \begin{equation}\label{e:claim0}
    z(\varphi)>\eta(\varphi), \qquad \varphi \in (\orho-\delta, \orho).
    \end{equation}
    By \eqref{e:claim0} and the definition of $\eta$ we have
$$
-\varepsilon <\frac{z(\varphi)}{D(\varphi)}<0, \qquad \varphi \in (\orho-\delta, \orho),
$$
and then
$$
\lim_{\varphi \to \orho^{\, -}}\frac{z(\varphi)}{D(\varphi)}=0.
$$
As in case {\em (ii)}, this means that $\varphi$ is a \emph{classical} semi-wavefront profile from $\orho$.

\smallskip
{\em (iv)} \emph{Case $c-h(\orho)=\dot D(\orho)=0$.} In this borderline case we show by an example that $\phi$ can be either classical or sharp. We consider the special case $h(\varphi)\equiv 0$; then the solution to problem \eqref{e:fo} with $c=0$ can be explicitly computed and is
$$
z(\varphi)=-\sqrt{2\int_{\varphi}^{\orho}D(s)g(s)\, ds}, \qquad \varphi \in [0, \orho].
$$
We further assume that
\begin{equation}\label{e:vanishingDg}
D(\varphi)=(\orho - \varphi)^{\alpha}, \quad g(\varphi)=(\orho - \varphi)^{\beta}, \quad \text{with } \alpha>1, \, \beta>0.
\end{equation}
In particular, we require $\alpha>1$ in order that $\dot D(\orho)=0$. Then we obtain that
$$
z(\varphi)=-\mu(\orho - \varphi)^{\frac{\alpha+\beta+1}{2}}, \quad \hbox{ for }\varphi\in(0, \orho) \hbox{ and } \mu:=\sqrt{\frac{2}{\alpha+\beta+1}}.
$$
We deduce
\begin{align*}
\lim_{\varphi \to \orho^{\, -}}\frac{z(\varphi)}{D(\varphi)} & = -\mu \lim_{\varphi \to \orho^{\, -}}(\orho - \varphi)^{\frac{\alpha+\beta+1}{2}-\alpha} = -\mu\lim_{\varphi \to \orho^{\, -}}(\orho - \varphi)^{\frac{-\alpha+\beta+1}{2}}
\\
&=\left\{
\begin{array}{ll}
0 & \hbox{ if }\alpha -\beta<1,
\\
-(\beta+1)^{-1/2} & \hbox{ if }\alpha-\beta=1,
\\
-\infty & \hbox{ if }\alpha-\beta>1.
\end{array}
\right.
\end{align*}
As a consequence, the corresponding semi-wavefront solution has a {\em classical} profile in the first case and a {\em sharp} profile in the remaining two cases.

\smallskip

Now, we assume (${\rm \tilde D}$). We can argue as above because of the existence result of Lemma \ref{l:Ex-sing2017}. The only difference consists in the computation of the limit in \eqref{e:quotient}. Indeed, since $\dot z(\orho)$ exists and is finite by Proposition \ref{p:slopePDnew}, it follows
\begin{equation}\label{eq:claim2}
\lim_{\xi \to \oxi ^+} \varphi^{\, \prime}(\xi)=\lim_{\varphi \to \orho^-}\frac{z(\varphi)}{D(\varphi)}=\lim_{\phi \to \orho^-}\frac{\frac{z(\phi)}{\phi-\orho}}{\frac{D(\phi)}{\phi-\orho}}=\frac{\dot z(\orho)}{-\infty}=0,
\end{equation}
where $\oxi $ is defined in \eqref{e:xi0-sharp}. Then every semi-wavefront profile is classical as in the proof above, case {\em (ii)}.

\smallskip

This concludes the proof in the case of semi-wavefronts {\em from} $\orho$. By the change of variables exploited in \cite[Theorem $2.7$]{Corli-Malaguti}, the existence of a semi-wavefront solution {\em to} $\orho$ with speed $c$ for \eqref{e:E} is equivalent to the existence of a semi-wavefront solution from $\orho$ and speed $-c$ of the equation $\rho_t-h(\rho)\rho_x=\left(D(\rho)\rho_x\right)_x+g(\rho)$, for $(x,t)\in\R\times[0,\infty)$.
\end{proofof}

\begin{remark}\label{rem:vanishing-order} The type of the profile in the critical case $c-h(\orho)=\dot D(\orho)=0$ also depends on the order of vanishing of $h(\rho)-c$ and not only on that of $D(\rho)$ and $g(\rho)$, as shown by an example in the proof above. More precisely, in addition to \eqref{e:vanishingDg}, we require
\begin{equation}\label{e:vanishingchz}
c-h(\phi)\sim(\orho-\phi)^\gamma,
\quad
z(\phi) \sim(\orho-\phi)^\delta,
\end{equation}
for some $\gamma>0$ and $\delta>1$ by Proposition \ref{p:slopePDnew}. Here $f\sim g$ means that $\lim_{\phi\to\orho}f(\phi)/g(\phi)$ is a non-zero real number. While the former expression in \eqref{e:vanishingchz} is simply an assumption, the latter should be proved; as a consequence, our analysis is merely formal. By \eqref{e:fo} we deduce $(\orho-\phi)^{\delta-1} \sim (\orho-\phi)^\gamma + (\orho-\phi)^{\alpha+\beta-\delta}$ for $\phi\to\orho-$.

If $\gamma \ge\alpha +\beta-\delta$, then $\delta-1= \alpha+\beta-\delta$. This implies $\delta = \frac12(\alpha+\beta+1)$ and in turn $\gamma\ge \frac12(\alpha+\beta-1)$. Then
\[
\frac{z(\phi)}{D(\phi)}\sim (\orho-\phi)^{\frac12(\beta-\alpha+1)}
\]
and the discussion is as in case {\em (iv)} of the proof of Theorem \ref{t:semi}.

If $\gamma < \alpha +\beta-\delta$, however, then $\delta-1= \gamma$ and $\gamma<\frac12(\alpha+\beta-1)$. Then
\[
\frac{z(\phi)}{D(\phi)}\sim (\orho-\phi)^{\gamma-\alpha+1}
\]
so that the discussion is analogous to that of the previous case but with $\gamma$ replacing $\beta$. Therefore, in this case, $\phi$ is classical if $\alpha-\gamma<1$ and sharp if $\alpha-\gamma\ge1$.
\end{remark}

We conclude this section by proving Theorem \ref{t:meno}.

\begin{proofof}{Theorem \ref{t:meno}}
Consider the equation
 \begin{equation}\label{e:Emeno}
\rho_t +  \hat f(\rho)_x=\left( \hat D(\rho)\rho_x\right)_x+ \hat g(\rho), \qquad t\ge 0, \, x\in \R,
\end{equation}
 where
 $$
 \hat f(\rho) = -f(\orho -\rho), \quad \hat D(\rho)=D(\orho -\rho), \quad\hat g(\rho) = -g(\orho -\rho), \quad \hbox{ for } \rho \in [0, \orho].
 $$
We denote $\hat h(\rho):=\hat f'(\rho)= h(\orho-\rho)$; remark that $\hat g$ satisfies (g) and $\hat D \in C^1[0, \orho]$.

We notice that $\phi$ is a classical solution of \eqref{e:tws} in $I$ with speed $c$ if and only if $\psi(\xi) = \orho-\phi(\xi)$ is a classical solution of
 \begin{equation}\label{e:twsmeno}
\left(\hat D(\psi)\psi^{\, \prime}\right)^{\, \prime} + \left(c-\hat h(\psi)\right)\psi^{\, \prime}+\hat g(\psi)=0
\end{equation}
in $I$ for the same speed $c$. As a consequence, equation \eqref{e:Emeno} has a  \emph{classical} semi-wavefront solution from $\orho$ if and only if \eqref{e:E} has a classical  semi-wavefront solution from $0$. If $D(0)=0$, i.e. if $\hat D(\orho)=0$, then \emph{sharp} semi-wavefronts from $\orho$ for \eqref{e:Emeno} appear (see Theorem \ref{t:semi}) and their profiles $\psi$ solve equation \eqref{e:twsmeno}. By Proposition \ref{l:sharp} these profiles satisfy
$$
\lim_{\xi \to \overline \xi^+}\hat D\left(\psi(\xi)\right)\psi^{\prime}(\xi)=0,
$$
where
$\overline{\xi} := \inf\left\{\xi <\varpi \, :\, \psi(\xi)<\orho\right\}=\inf\left\{\xi <\varpi \, :\, \varphi(\xi)>0\right\}$.
Since
$$
\lim_{\xi \to \overline \xi^+}D\left(\varphi(\xi)\right)\varphi^{\prime}(\xi) = -\lim_{\xi \to \overline \xi^+}\hat D\left(\psi(\xi)\right)\psi^{\prime}(\xi),
$$
then $\varphi$ is a sharp semi-wavefront profile of \eqref{e:E} from $0$ with the same $c$. The converse implication is also true. An analogous discussion is valid for semi-wavefront solutions to $0$.

At last, notice that the semi-wavefront solutions of \eqref{e:Emeno} are completely described in \cite{Corli-Malaguti} and Theorem \ref{t:semi} above. The theorem is proved.
\end{proofof}
\section{Strictly monotone solutions}\label{s:nature-1}
\setcounter{equation}{0}

In this section we prove Theorem \ref{t:strictly}.

\begin{proofof}{Theorem \ref{t:strictly}}
We only prove the result in the case of semi-wavefronts from $\overline \rho$; for semi-wavefronts to $\overline \rho$ the result is deduced as in the proof of Theorem \ref{t:semi}. We assume without any loss of generality that the estimates on $g$ in the statement hold in the whole interval $[0,\orho]$.

\smallskip
{\em (i)} We assume $g(\rho)\le L(\overline \rho-\rho)$ and $c>h(\overline\rho)$.  In fact, both here and in the following item, in the case $c<h(\orho)$ the profile is sharp and then it is non-strictly monotone.
Let $\varphi$ be a semi-wavefront in $(-\infty, \varpi)$ with speed $c$ and denote by $z(\varphi)$, $\varphi \in [0, \overline \rho]$, the solution of \eqref{e:fo} with the same wave speed $c$. For $n\in\N$ we define
\begin{equation}\label{eq:def-di-eta}
\eta_n(\phi)=D(\phi)\left[a\left(\phi-\orho\right)-\frac{1}{n}\right],
\end{equation}
where the constant $a$ is chosen to satisfy
\begin{equation}
\label{eq:def-a}
a >\frac{L}{c-h(\orho)}>0.
\end{equation}
We claim that there exist $\overline{n}>0$ and $\delta>0$, which only depends on $\overline{n}$, such that for every $n>\overline{n}$ we have that $\eta_n(z)$ is a strict upper-solution of \eqref{e:zeq} in $[\orho-\delta, \orho)$. This amounts to show that (see \cite[Definition $4.2$]{Corli-Malaguti})
\begin{equation}\label{eq:hjk1}
\dot\eta_n(\phi)- h(\phi) + c +\frac{D(\phi)g(\phi)}{\eta_n(\phi)}>0, \quad \phi \in (\orho-\delta, \orho).
\end{equation}
We consider \eqref{eq:hjk1}; by the assumption on $g$ we deduce
\begin{align}\label{eq:kl1}
\notag
&-\left(-h(\phi) + c + \frac{D(\phi)g(\phi)}{\eta_n(\phi)}\right) -\dot\eta_n(\phi)\\
&\qquad= -\left(-h(\phi) + c + \frac{g(\phi)}{a(\phi-\orho)-\frac{1}{n}}\right) +\dot D(\phi)\left[a\left(\orho -\phi\right) +\frac{1}{n}\right] - aD(\phi)\\
&\qquad < h(\phi) - c +  \frac{g(\phi)}{a(\orho -\phi)+\frac{1}{n}}+\dot D(\phi)\left[a\left(\orho -\phi\right) +\frac{1}{n}\right]\notag\\
&\qquad < h(\phi) - c + \frac{L(\orho -\phi)}{a\left(\orho - \phi + \frac{1}{an}\right)}+\dot D(\phi)\left[a\left(\orho -\phi\right) +\frac{1}{n}\right]\notag\\
&\qquad < h(\phi) - c +\frac{L}{a} + \dot D(\phi)\left[a\left(\orho -\phi\right) +\frac{1}{n}\right]\notag.
\end{align}
By \eqref{eq:def-a} we have that
\begin{equation*}
\lim_{\genfrac{}{}{0pt}{}{\phi\to\orho^-}{n\to\infty}
} \left( h(\phi) - c +\frac{L}{a} + \dot D(\phi)\left[a\left(\orho -\phi\right) +\frac{1}{n}\right]\right) =   h(\orho) - c +\frac{L}{a} <0.
\end{equation*}
Hence, for every $\eps>0$ there exist $\overline{n}>0$ and $\delta>0$ such that
\begin{equation*}
h(\phi) - c +\frac{L}{a} + \dot D(\phi)\left[a\left(\orho -\phi\right) +\frac{1}{n}\right] <-\eps^2 <0,\quad \textrm{for $n>\overline{n}$,  $\phi\in (\orho - \delta , \orho)$,}
\end{equation*}
which implies \eqref{eq:hjk1} and, then, our claim.

\smallskip
By \eqref{e:z-slope} we have
\[
\lim_{\phi\to \overline \rho^-} \frac{z(\phi)-z( \overline \rho)}{\phi- \overline \rho} =\dot z(\orho)=0.
\]
Consider a sequence $\{\psi_p\}\subset (0, \orho)$ converging to $\orho$. Thanks to \eqref{e:fo} and  the Mean Value Theorem, for every $p\in\N$ there exists $\theta_p\in(\psi_p,\orho)$  such that
\begin{equation*}
\dot{z}(\theta_p)=\frac{z(\psi_p)}{\psi_p -\orho}
\end{equation*}
and hence
\begin{equation}\label{eq:theta-n-2}
\theta_p\to \orho, \quad \dot{z}(\theta_p)\to 0.
\end{equation}
By substituting $\theta_p$ in \eqref{e:zeq}, we obtain
\begin{equation}\label{eq:z-theta-7}
\frac{D(\theta_p)g(\theta_p)}{z(\theta_p)}= h(\theta_p)-c- \dot z(\theta_p).
\end{equation}
By {\rm (g)}, the estimate on $g$, \eqref{eq:theta-n-2}, and \eqref{eq:z-theta-7} we find
\[
\lim_{p\to\infty}\frac{D(\theta_p)}{z(\theta_p)}=- \lim_{p\to\infty}\frac{c-h(\theta_p)+\dot z(\theta_p)}{g(\theta_p)}=-\infty,
\]
and then we can find a subsequence $\{\theta_{p_n}\}$ such that
\begin{equation}\label{e:Sarri}
D(\theta_{p_n})>-nz(\theta_{p_n}),
\end{equation} for $n\in \N$. We can assume  $\theta_{p_n}\in (\orho-\delta, \orho)$ for all $n>\overline{n}$; by \eqref{eq:def-di-eta} and \eqref{e:Sarri} we deduce
\[
\frac{\eta_n(\theta_{p_n})}{z(\theta_{p_n})} = \frac{D(\theta_{p_n})\left[a(\orho - \theta_{p_n}) +\frac{1}{n}\right]}{-z(\theta_{p_n})}
> n\left[a(\orho - \theta_{p_n}) +\frac{1}{n}\right] >1, \quad \mbox{ for }n \in \mathbb{N}.
\]
Hence
\begin{equation}\label{eq:eta-z}
\eta_n(\theta_{p_n}) < z(\theta_{p_n}) \quad \mbox{ for }n \in \mathbb{N}.
\end{equation}
Thanks to \eqref{eq:hjk1} and \eqref{eq:eta-z}, we can apply  \cite[Lemma $4.3$ \emph{(2.ii)}]{Corli-Malaguti} to the interval $(\orho-\delta,\theta_{p_n}]$, because $\eta(\phi)<0$, and  conclude that
\begin{equation}\label{e:eta n}
z(\phi)>\eta_n(\phi), \qquad \text{for }\phi\in (\orho-\delta, \theta_{p_n}) \text{ and } n>\overline{n}.
\end{equation}
Let $\xi(\phi)$ be the inverse function of the function $\phi$ in $[0,\orho)$, see \cite[Remark 6.3]{Corli-Malaguti}, and $\overline{\xi}$ as in \eqref{e:xi0-sharp}; then, by \eqref{e:eta n}, the definitions of $z$ and $\eta_n$ we have
\begin{align*}
\overline{\xi} - \xi\left(\orho-\delta\right) &=\lim_{n\to\infty} \int_{\orho-\delta}^{\theta_{p_n}}\xi^{\, \prime}(\varphi)\, d\varphi
 = \lim_{n\to\infty}\int_{\orho-\delta}^{\theta_{p_n}}\frac{1}{\varphi^{\, \prime}\left(\xi(\varphi)\right)}\, d\varphi
 \\
& =\lim_{n\to\infty}\int_{\orho-\delta}^{\theta_{p_n}}\frac{D(\varphi)}{z(\varphi)}\, d\varphi
< \lim_{n\to\infty}\int_{\orho-\delta}^{\theta_{p_n}}\frac{D(\varphi)}{\eta_n(\varphi)}\, d\varphi\\
&= \lim_{n\to\infty}\int_{\orho-\delta}^{\theta_{p_n}}\frac{1}{a\left(\varphi-\overline \rho\right)-\frac {1}{n}}\, d\varphi
\\
&=\frac{1}{a} \lim_{n\to\infty}\ln\left\vert \frac{a\left(\theta_{p_n} -\orho\right) -\frac{1}{n}}{-a\delta-\frac{1}{n} }\right\vert =-\infty.
\end{align*}
Hence, $\overline{\xi}=-\infty$.

\smallskip

{\em (ii)} We assume $g(\rho)\ge L(\overline \rho-\rho)^{\alpha}$ for $\alpha\in (0, 1)$ and $c \ge h(\orho)$. We fix a value $\beta \in (\frac{\alpha+1}{2}, 1)$ and define
\begin{equation}\label{eq:hmsi}
\overline h:=\min_{\varphi \in [\orho/2, \overline \rho]}\left(h(\varphi)-c\right), \quad -\sigma^2:=\min_{\varphi \in [\orho/2, \overline \rho]}\dot D(\phi), \quad M:=\max_{\varphi \in [\orho/2, \overline \rho]}D(\phi).
\end{equation}
For every $n \in \mathbb{N}$ with $\frac{\overline\rho}{2}<\overline \rho -\frac 1n$, we introduce the function $\omega_n \colon [\frac{\overline\rho}{2}, \overline \rho] \to \mathbb{R}$ defined by  \begin{equation*}
  \omega_n(\varphi) =\left\{\begin{array}{ll}-kD(\phi)(\overline \rho -\frac 1n -\varphi)^{\beta} & \hbox{ if }\varphi\in  [\frac{\overline\rho}{2},
  \overline \rho -\frac 1n],\\[1mm]
  0 & \hbox{ if } \varphi\in  (\overline \rho -\frac 1n,\orho],
  \end{array}
  \right.
  \end{equation*}
where $k$ is a positive constant, see Figure \ref{f:omegan}. By choosing a suitable $k$, we claim that
\begin{equation}\label{e:claim}
\omega_n(\varphi)\ge z(\varphi), \qquad \varphi\in  [\overline\rho/2,
  \overline \rho].
\end{equation}


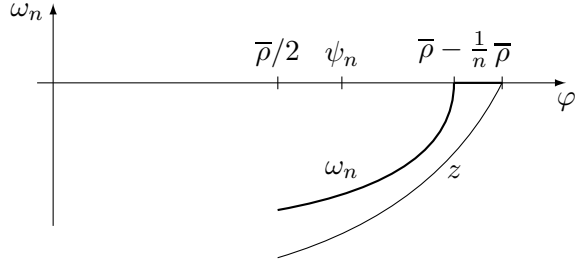
\begin{figure}[htbp]
\begin{picture}(100,100)(-80,-10)
\setlength{\unitlength}{1.2pt}

\put(50,50){
\put(0,0){\vector(1,0){160}}
\put(0,0){\line(-1,0){5}}
\put(160,-3){\makebox(0,0)[t]{$\phi$}}
\put(0,-45){\vector(0,1){70}}
\put(-3,25){\makebox(0,0)[rt]{$\omega_n$}}

\put(140,-2){\line(0,1){4}}
\put(140,5){\makebox(0,0)[b]{$\orho$}}
\put(90,-2){\line(0,1){4}}
\put(90,5){\makebox(0,0)[b]{$\psi_n$}}
\put(125,-2){\line(0,1){4}}
\put(125,5){\makebox(0,0)[b]{$\orho-\frac1n$}}
\put(70,-2){\line(0,1){4}}
\put(70,5){\makebox(0,0)[b]{$\orho/2$}}

\put(0,0){\thicklines{\qbezier(70,-40)(125,-30)(125,0)}}
\put(125,0){\thicklines{\line(1,0){15}}}
\put(90,-30){\makebox(0,0)[b]{$\omega_n$}}

\put(0,0){\qbezier(70,-55)(120,-40)(140,0)}
\put(125,-30){\makebox(0,0)[b]{$z$}}

}
\end{picture}
\caption{\label{f:omegan}{The function $z$ (thin line) and the upper-solution $\omega_n$ (thick line).}}
\end{figure}


Indeed, since $z(\varphi)<0$ for $\phi\in(0, \overline \rho)$, by a continuity argument we can find $\psi_n \in (\frac{\overline \rho}{2}, \overline \rho -\frac 1n )$ such that $\omega_n(\varphi)\ge z(\varphi)$ in $[\psi_n, \overline \rho]$. If we show that $\omega_n$ is a strict lower-solution of \eqref{e:zeq} on $[\frac{\overline \rho}{2}, \psi_n]$, then we can apply \cite[Lemma $4.3$\emph{(2.i)}]{Corli-Malaguti} in the interval $(\frac{\overline \rho}{2}, \psi_n]$, because $\omega_n(\phi) < 0$ in $[\frac{\orho}{2},\orho-\frac1n]$, and prove \eqref{e:claim}.

By the assumption on $g$ and \eqref{eq:hmsi}, we
obtain, for $\varphi \in [\frac{\overline \rho}{2}, \psi_n]$,
\begin{align}
\lefteqn{h(\varphi)-c-\frac{D(\varphi)g(\varphi)}{\omega_n(\varphi)}
=
h(\varphi)-c+\frac{D(\varphi)g(\varphi)}{kD(\phi)(\overline \rho -\frac 1n -\varphi)^{\beta}}}\nonumber
\\
\ge &\,\overline h +\frac{ L(\overline \rho -\varphi)^{\alpha}}{k(\overline \rho -\frac 1n -\varphi)^{\beta}}\nonumber
=\overline h +\frac{ L(\overline \rho -\varphi)^{\alpha}}{k(\overline \rho -\frac 1n -\varphi)^{\alpha}}\cdot \frac{1}{(\overline \rho -\frac 1n -\varphi)^{\beta-\alpha}}\nonumber
\\
\ge\,&
\overline h + \frac{L}{k}\frac{1}{(\overline \rho -\frac 1n -\varphi)^{\beta-\alpha}}.\label{e:second}
\end{align}
Now, we introduce the function $\gamma_n \colon [\frac{\overline\rho}{2}, \psi_n ] \to \mathbb{R}$ defined by
\begin{equation*}
 \gamma_n(\varphi)=\overline h + \frac{L}{k}\frac{1}{(\overline \rho -\frac 1n -\varphi)^{\beta-\alpha}}-\dot \omega_n(\varphi).
\end{equation*}
We must show that $\gamma_n(\phi)\ge0$. Since $0<\orho-\frac1n -\phi<\frac{\orho}{2}$  for $\varphi \in [\frac{\overline \rho}{2}, \psi_n]$,
it follows from \eqref{eq:hmsi} that, for $\phi\in[\frac{\orho}{2},\psi_n)$,
\begin{align}
\notag
&\quad\gamma_n(\varphi) = \overline h + \frac{L}{k}\frac{1}{(\overline \rho -\frac 1n -\varphi)^{\beta-\alpha}} + k\dot D(\phi)
 \left(\orho -\frac{1}{n} -\phi\right)^{\beta}\notag
 \\
 &\qquad-k\beta D(\phi) \left(\orho -\frac{1}{n} -\phi\right)^{\beta-1}
 \notag
 \\
& \ge \overline h + \frac{L}{k}\frac{1}{(\overline \rho -\frac 1n -\varphi)^{\beta-\alpha}} -k\sigma^2\left(\orho -\frac{1}{n} -\phi\right)^{\beta}-Mk\beta\left(\orho -\frac{1}{n} -\phi\right)^{\beta-1}\notag
\\
& >  \overline h -k\sigma^2 \left(\frac{\orho}{2}\right)^{\beta}+ \frac{1}{(\overline \rho -\frac 1n -\varphi)^{\beta-\alpha}}\left[\frac{L}{k} - Mk\beta\left(\orho -\frac{1}{n} -\phi\right)^{2\beta-(1+\alpha)}\right]\notag
\\
& \ge \overline h -k\sigma^2 \left(\frac{\orho}{2}\right)^{\beta} + \frac{1}{\left(\frac{\orho}{2}\right)^{\beta-\alpha}}\left[\frac{L}{k} - Ak\right],
\label{eq:hjkl}
\end{align}
where $A=M\beta(\orho/2)^{2\beta-(1+\alpha)}$. If $k$ is sufficiently small, then the right hand side of \eqref{eq:hjkl} is positive and then $\gamma_n(\phi)>0$. Hence, $\omega_n$ is a strict lower-solution of \eqref{e:zeq} on $(\frac{\overline\rho}{2}, \psi_n]$ and claim \eqref{e:claim} is proved.

The sequence $\{\omega_n\}_n$ is  monotone and therefore we can define
\begin{equation*}
\omega(\varphi):=\lim_{n \to \infty}\omega_n(\varphi)=-kD(\phi)(\overline \rho-\varphi)^{\beta}, \quad \varphi
\in [\overline\rho/2,
  \overline \rho].
\end{equation*}
By \eqref{e:claim} we have $\omega(\varphi)\ge z(\varphi)$ for $\varphi \in [\frac{\overline\rho}{2}, \overline \rho ]$. Let $\xi(\phi)$ be as in \emph{(i)} and $\overline{\xi}$ as in \eqref{e:xi0-sharp}. It follows from the definition of $z$ that
\begin{align*}
\overline \xi  - \xi\left(\frac{\overline \rho}{2}\right) &= \int_{\frac{\overline \rho}{2}}^{\overline \rho}\xi^{\, \prime}(\varphi)\, d\varphi
 = \int_{\frac{\overline \rho}{2}}^{\overline \rho}\frac{1}{\varphi^{\, \prime}\left(\xi(\varphi)\right)}\, d\varphi=\int_{\frac{\overline \rho}{2}}^{\overline \rho}\frac{D(\varphi)}{z(\varphi)}\, d\varphi
\\
& \ge \int_{\frac{\overline \rho}{2}}^{\overline \rho}\frac{D(\varphi)}{\omega(\varphi)}\, d\varphi=  -\frac{1}{k}\int_{\frac{\overline \rho}{2}}^{\overline \rho}\frac{1}{(\overline \rho-\varphi)^{\beta}}\, d\varphi=-\frac{1}{k(1-\beta)}\left(\frac{\overline \rho}{2} \right)^{1-\beta}.
\end{align*}
Therefore $\overline \xi \in \mathbb{R}$ and so $\varphi(\xi)\equiv \overline \rho$ for $\xi \le \overline \xi$.
\end{proofof}

\section{Convergence of semi-wavefronts to wavefronts}\label{S:convergence}
\setcounter{equation}{0}

In this section we prove Theorem \ref{t:convergence}. Again the discussion is based on the first-order problem \eqref{e:fo}; solutions to this problem are provided in Lemma \ref{l:Ex-sing2017} under (g) and in \cite[Theorem 2.2]{MMconv} under (g${}_0$). In particular, if $g=g_1$ satisfies (g${}_0$) then problem \eqref{e:fo} admits solutions with $z(0)=0$ if and only if $c\ge c_1^*$; only these solutions are interesting since they correspond to wavefront profiles. If $g=g_2$ satisfies (g) then problem \eqref{e:fo} admits solutions for every $c\in\R$ and all of them correspond to semi-wavefront profiles.

The following lemma gives a comparison result between solutions of \eqref{e:fo}. For simplicity, we do not state it in full generality but focus on what we need below.

\begin{lemma}\label{l:mon2} Assume ${\rm (\hat D)}$ or ${\rm (\tilde D)}$ and let $g_1\le g_2$ be continuous functions in $[0,\orho]$ 
with $g_1$ satisfying either {\rm (g${}_0$)} or {\rm (g)} and $g_2$ satisfying  {\rm (g)}. Let $c_1,c_2\in\R$ with $c_2\le c_1$ and $c_1^*\le c_2$ when  $g_1(0)=0$. Denote with $z_i$ the solution of problem \eqref{e:fo} with $c=c_i$ and $g=g_i$, $i=1,2$.  Then
\begin{equation}\label{e:comparison}
z_1(\phi)\ge z_2(\phi) \quad \mbox{ for every } \phi\in [0,\orho].
\end{equation}
\end{lemma}
\begin{proof} 
We split the proof into two cases.

\begin{enumerate}[{\em (i)}]

\item Assume $c_1>c_2$.  In this case we prove a bit more than \eqref{e:comparison}, namely we claim
\[
z_1(\phi) > z_2(\phi) \quad \mbox{ for every } \phi\in (0,\orho).
\]
If there exists $\phi_0  \in (0, \orho)$ such that $z_1(\phi_0)\le z_2(\phi_0)$, we have
\begin{equation*}
\dot z_1(\phi_0) = h(\phi_0)-c_1+\frac{D(\phi_0)g_1(\phi_0)}{-z_1(\phi_0)}<h(\phi_0)-c_2+\frac{D(\phi_0)g_2(\phi_0)}{-z_2(\phi_0)} = \dot z_2(\phi_0).
\end{equation*}
Hence $z_2(\phi) > z_1(\phi)$ for $\phi$ in a right neighborhood of $\phi_0$; by repeating the same reasoning we arrive to the contradictory conclusion that  $z_2(\orho)>z_1(\orho)=0$. This proves the claim.

\item  Assume $c_1=c_2$. Let $\{\hat c_n\}_{n\ge1}$ be a sequence decreasing to $c_1=c_2$ and denote with $\zeta_n$ the corresponding solutions of \eqref{e:fo} with $c=\hat c_n$ and $g=g_1$; we emphasize that $\zeta_n$ exists either because of Lemma \ref{l:Ex-sing2017} (if $g_1$ satisfies {\rm (g)}) or by   \cite[Theorem 2.2]{MMconv} (if $g_1$ satisfies {\rm (g${}_0$)}). By \emph{(i)} we deduce that $\zeta_n> z_2$ in $(0,\orho)$, for all $n\ge1$ and  that $\{\zeta_n\}_{n\ge1}$ is a decreasing sequence. Let
$z_1(\phi):=\displaystyle{\lim_{n \to\infty}} \zeta_n (\phi)$, for $\phi \in [0, \orho]$. By the Monotone Convergence Theorem it is easy to show that $z_1$ is a solution of $\dot z = h(\phi)-c_1-\frac{D(\phi)g_1(\phi)}{z(\phi)}$ with $z_1(\orho)=0$ and $z_1(0)=0$ if $g_1(0)=0$; by applying Lemma \ref{l:Ex-sing2017} or \cite[Theorem 2.2]{MMconv} according that $g_1$ satisfies {\rm (g)} or {\rm (g${}_0$)}, we obtain that $z_1$ is the (unique) solution of \eqref{e:fo} when $c=c_1$ and $g=g_1$. At last, the estimate follows since $\zeta_n> z_2$ in $(0,\orho)$ for all $n\ge1$ implies that $z_1\ge z_2$ there and the proof is complete.\qedhere
\end{enumerate}
\end{proof}

For $n \ge 0$ we denote with $(P_n)$ the first-order problem \eqref{e:fo} corresponding to $g=g_n$ and to $c\ge c_0^*$. The following proposition deals with the convergence of $\{z_n\}_{n\ge1}$.
\begin{proposition} \label{p:conv zn} Assume ${\rm (\hat D)}$ or ${\rm (\tilde D)}$ and let $g_0$ and $\{g_n\}_{n\ge1}$ be as in  Theorem  \ref{t:convergence}. Let $z_n$ be the solution of ($P_n$), with $ n\ge 0$, for a  given  $c\ge c_0^*$.

Then, the sequence $\{z_n\}_{n\ge1}$ is increasing, satisfies $z_n(\phi)\le z_0(\phi)$ for $\phi \in [0, \orho]$ and $z_n\to z_0$ uniformly in every interval $[a,b]\subset (0, \orho)$.
\end{proposition}
\begin{proof}
The first two claims follow from Lemma \ref{l:mon2} and we are left to the convergence.
Define
\begin{equation*}
\tilde z(\phi):=\lim_{n\to\infty}z_n(\phi), \qquad \phi \in [0, \orho].
 \end{equation*}
By the monotonicity of $\{z_n\}_{n\ge1}$ we have that $\tilde z\le z_0$; we want to show that $\tilde z = z_0$. In every interval $[a,b] \subset (0, \orho)$ we obtain
 \begin{equation}\label{eq:dominated}
0<\frac{D(\phi)g_n(\phi)}{-z_n(\phi)}\le \frac{D(\phi)g_1(\phi)}{-z_0(\phi)} \le \frac{D(\phi)\displaystyle{\max_{\phi \in [a,b]}}g_1(\phi)}{-\displaystyle{\max_{\phi \in [a,b]}}z_0(\phi)}, \quad \phi \in (0, \orho).
 \end{equation}
This implies that the convergence
 \begin{equation*}
\frac{D(\phi)g_n(\phi)}{-z_n(\phi)} \to \frac{D(\phi)g_0(\phi)}{-\tilde z(\phi)}
 \end{equation*}
 is dominated in  $[a,b]$. So, we can use the Dominated Convergence Theorem to show that $\tilde z$ solves the same problem \eqref{e:fo}
of $z_0$. If $\tilde z(0)=0$, then $\tilde z=z_0$ by \cite[Theorem 2.2]{MMconv}. Now we prove that the remaining case $\tilde z(0)< 0$ is not possible and hence the convergence is proved. We reason as in the proof of case {\em (c)} in \cite[Theorem 2.6]{Corli-Malaguti}. More precisely, $\tilde z(0)<0=z_0(0)$ implies the existence of $\delta \in (0, \orho)$ such that $\tilde z(\phi)<z_0(\phi)$ in $(0, \delta)$ and  then
$$
\dot z_0(\varphi)-\dot {\tilde z}(\varphi) =\frac{D(\varphi)g_0(\varphi)}{-z_0(\varphi)}-\frac{D(\varphi)g_0(\varphi)}{- \tilde z(\varphi)}>0, \quad \mbox{for }\varphi \in (0, \delta).
$$
Then, it cannot happen that $\tilde z(\phi_0)=z_0(\phi_0)$ for some $\phi_0 \in [\delta, \orho)$; hence $\tilde z <z_0$ in $[0, \orho)$ and then $z_0-\tilde z$ is strictly increasing in $[0,\overline{\rho})$. This implies$$
\lim_{\varphi \to \overline \rho^{\, -}} \left( z_0(\varphi)- \tilde z(\varphi) \right) > z_0(0)-\tilde z(0)>0,
$$
in contradiction with $\tilde z (\overline \rho^{\, -})=z_0(\overline \rho^{\, -})=0$.
At last, again by \eqref{eq:dominated}, we have that the convergence is uniform in every interval $[a,b]\subset (0, \orho)$.
\end{proof}
\begin{proofof}{Theorem \ref{t:convergence}} The existence of the solutions $\phi_n$ is guaranteed 
by Theorem \ref{t:semi} when either (${\rm \tilde D}$) or (D) hold; in the remaining case, namely when
(${\rm \hat D}$) holds and $D(\orho)>0$, the existence of $\phi_n$ follows by \cite[Theorem 2.7]{Corli-Malaguti}. About the existence of $\phi_0$ we refer to \cite{MMconv}; we point out when $D(0)=\dot D(0)=0$ and/or $D(\orho)=\dot D(\orho)=0$, the existence of $\phi_0$ can be proved under further assumptions on $D$ and $g$, see \cite{MMconv}.

Assume $c \ge c_0^*$. Arguing as in the proof of Theorem \ref{t:semi}, it is easy to show that $\phi_n$, for $n \ge 0$, is the {\em classical} solution  of the initial-value problem
\begin{equation}\label{e:Cauchy}
\left\{
\begin{array}{ll}
&\phi^{\prime}=\frac{z_n(\phi)}{D(\phi)},
\\
&\phi(0)=\frac{\orho}{2},
\end{array}
\right.
\end{equation}
 on its maximal existence interval $(\sigma_n, \tau_n)$, with $-\infty \le \sigma_n<0<\tau_n\le
 +\infty$, i.e., $\phi_n(\sigma_n^+)=\orho$, $\phi_n(\tau_n^-)=0$. Let $[\alpha, \beta]\subset (\sigma_0, \tau_0)$ be a fixed
interval such that $\phi_0([\alpha, \beta])=[a,b] \subset (0,\orho)$. According to
Proposition \ref{p:conv zn}, we have that $z_n\to z_0$ uniformly in $[a,b]$. By the continuous dependence of the solution on a parameter \cite[Ch. 2, Theorem 4.1]{Coddington-Levinson}, $\phi_n$ is defined on $[\alpha, \beta]$ for a sufficiently large $n$ and $\phi_n\to\phi_0$ uniformly on $[\alpha, \beta]$.

Moreover
\[
\phi_n^{\prime}(\xi) =\frac{z_n\left(\phi_n(\xi)\right)}{D\left(\phi_n(\xi)\right)} \to \frac{z_0\left(\phi_0(\xi)\right)}{D\left(\phi_0(\xi)\right)}=\phi_0^{\prime}(\xi)
\]
uniformly on $[\alpha, \beta]$. This completes the proof.
\end{proofof}

The last result of this section concerns some comparison properties, see Figure \ref{f:phi-monotone}.


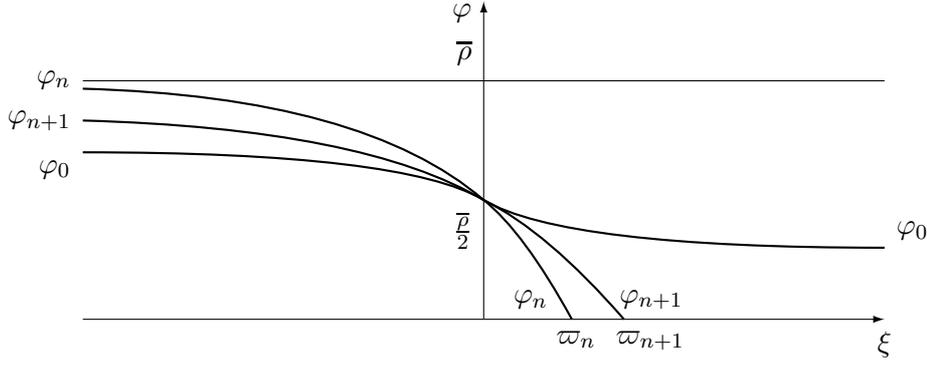
\begin{figure}[htbp]
\begin{picture}(100,100)(-80,-10)
\setlength{\unitlength}{1.5pt}

\put(90,0){
\put(0,0){\vector(1,0){100}}
\put(0,0){\line(-1,0){100}}
\put(0,60){\line(1,0){100}}
\put(0,60){\line(-1,0){100}}
\put(100,-3){\makebox(0,0)[t]{$\xi$}}
\put(0,0){\vector(0,1){80}}
\put(-3,77){\makebox(0,0)[r]{$\phi$}}
\put(-3,67){\makebox(0,0)[r]{$\overline{\rho}$}}
\put(-3,27){\makebox(0,0)[rt]{$\frac{\overline{\rho}}{2}$}}

\put(0,0){\thicklines{\qbezier(-100,42)(-20,42)(0,30)}}
\put(0,0){\thicklines{\qbezier(0,30)(20,18)(100,18)}}
\put(-103,40){\makebox(0,0)[rt]{$\phi_0$}}
\put(103,20){\makebox(0,0)[lb]{$\phi_0$}}

\put(0,0){\thicklines{\qbezier(-100,58)(-30,56)(0,30)}}
\put(0,0){\thicklines{\qbezier(0,30)(10,22)(22,0)}}
\put(-103,58){\makebox(0,0)[br]{$\phi_n$}}
\put(16,5){\makebox(0,0)[r]{$\phi_n$}}
\put(28,-3){\makebox(0,0)[tr]{$\varpi_n$}}

\put(0,0){\thicklines{\qbezier(-100,50)(-30,48)(0,30)}}
\put(0,0){\thicklines{\qbezier(0,30)(16,22)(35,0)}}
\put(-103,50){\makebox(0,0)[r]{$\phi_{n+1}$}}
\put(34,5){\makebox(0,0)[l]{$\phi_{n+1}$}}
\put(33,-3){\makebox(0,0)[tl]{$\varpi_{n+1}$}}
}

\end{picture}
\caption{\label{f:phi-monotone}{The profiles in the case $g_0(\phi)<g_n(\phi)$ and $g_n(\phi)>g_{n+1}(\phi)$} for every $\phi\in[0,\orho]$.}
\end{figure}

\begin{corollary}\label{c:last}
Under the hypotheses of Theorem \ref{t:convergence} assume $g_0(\phi)<g_n(\phi)$ for $\phi \in (0, \orho)$. Then
$$
\phi_0(\xi)<\phi_n(\xi) \ \hbox{ if }\xi<0 \quad \text{and} \quad  \phi_0(\xi)>\phi_n(\xi)\ \hbox{ if }\xi \in (0, \varpi_n).
$$
Moreover, if the sequence $\{g_n\}_{n\ge1}$ is strictly decreasing, then $\{\phi_n\}_{n \ge 1}$ is decreasing in $(-\infty, 0)$ and increasing in every interval $[0, a]$ with $a<\varpi_n$ for all $n$.
\end{corollary}
\begin{proof} Reasoning as in the proof of Lemma \ref{l:mon2}\emph{(i)} we obtain
\begin{equation}\label{e:zn0S}
z_n<z_0, \quad \phi \in (0, \orho),
\end{equation}
for all $n \ge1$. We deduce that $\phi_0^{\prime}(0)=D(\frac{\orho}{2})z_0(\frac{\orho}{2})>D(\frac{\orho}{2})z_n(\frac{\orho}{2})=\phi_n^{\prime}(0)$, hence $\phi_0>\phi_n$ in a right neighborhood of $0$. Fix $n$ and define
 $$
 \oxi:=\sup\left\{ \xi \in(0,\varpi_n) \, : \, \phi_0(\xi)>\phi_n(\xi)\right\}.
 $$
If $\oxi <\varpi_n$, then $\phi_0>\phi_n$ in $[0, \oxi)$ and $\phi_0(\oxi)=\phi_n(\oxi):=\overline{\phi}$. We denote by $\xi_0$ and $\xi_n$ the inverse functions of $\phi_0$ and $\phi_n$ respectively; they exist because the profiles are strictly monotone owing to the sign condition on $z$. Then, by \eqref{e:zn0S},
\begin{align*}
D(\overline \varphi)\varphi_n^{\, \prime}(\overline \xi) & = D(\overline \varphi)\varphi_n^{\, \prime}\left(\xi_n(\overline \varphi)\right) = z_n(\overline \varphi)
\\
&<z_0(\overline \varphi) = D(\overline \varphi)\varphi_0^{\, \prime}\left(\xi_0(\overline \varphi )\right)=D(\overline \varphi)\varphi_0^{\, \prime}(\overline \xi).
\end{align*}
By the sign condition on $D$ we deduce that $\varphi_0^{\, \prime}(\overline \xi)>\varphi_n^{\, \prime}(\overline \xi)$, which contradicts $\phi_0>\phi_n$ in $[0, \oxi)$.  A similar  reasoning applies for $\xi<0$. The last statement follows analogously.
\end{proof}


\section{The case when $g$ changes sign}\label{s:gposneg}
In this section, we prove Theorem \ref{t:tws}. 

\begin{proofof}{Theorem \ref{t:tws}}
We begin by considering the interval $[0, \rho_0]$. By \cite[Theorem 2.7]{Corli-Malaguti},  with $\orho$ replaced by $\rho_0$, we deduce that, for any wave speed $c\in\R$, there exist strict classical semi-wavefront solutions both from $\rho_0$ and to $\rho_0$; moreover, they are unique (up to shifts) in the class of classical or sharp solutions. Their corresponding profiles are not strictly monotone, by \eqref{eq:goodg} and \cite[Theorem $2.9$]{Corli-Malaguti}, and satisfy,  see Figure \ref{f:4phi},
\begin{equation}\label{eq:uno}
\left\{
\begin{array}{lll}
\hbox{$\phi_{1,a}$ from $\rho_0$} 
&\left\{
\begin{array}{l}
\phi_{1,a}(\xi) \equiv \rho_0, \ \xi\in(-\infty,\oxi_{1,a}], 
\\
\phi_{1,a}(\xi)\in[0,\rho_0), \ \xi\in(\oxi_{1,a},\varpi_{1,a}],
\end{array}
\right.
\\[4mm]
\hbox{$\phi_{1,b}$ to $\rho_0$} 
&
\left\{
\begin{array}{l}  
\phi_{1,b}(\xi)\in[0,\rho_0), \ \xi\in[\varpi_{1,b}, \oxi_{1,b}),
\\
\phi_{1,b}(\xi)\equiv\rho_0, \ \xi\in[\oxi_{1,b},\infty),
\end{array}
\right.\end{array}
\right.
\end{equation}
for some $\oxi_{1,a}$ and $\oxi_{1,b}$.


\begin{figure}[htbp]
\begin{picture}(100,120)(-80,-10)
\setlength{\unitlength}{1pt}

\put(0,0){
\put(30,0){
\put(0,0){\vector(1,0){100}}
\put(0,0){\line(-1,0){100}}
\put(0,50){\line(1,0){100}}
\put(0,50){\line(-1,0){100}}
\put(0,100){\line(1,0){100}}
\put(0,100){\line(-1,0){100}}
\put(100,-3){\makebox(0,0)[tr]{$\xi$}}
\put(0,0){\vector(0,1){110}}
\put(-3,110){\makebox(0,0)[tr]{$\phi$}}
\put(-3,47){\makebox(0,0)[tr]{$\rho_0$}}
\put(-3,97){\makebox(0,0)[tr]{$\orho$}}

\put(0,0){
\put(0,0){\thicklines{\qbezier(-70,50)(-20,50)(-10,0)}}
\put(-70,50){\thicklines{\line(-1,0){30}}}
\put(-100,47){\makebox(0,0)[tl]{$\phi_{1,a}$}}
\put(-22,30){\makebox(0,0)[tr]{$\phi_{1,a}$}}
\multiput(-70,0)(0,5){10}{$.$}
\put(-70,-3){\makebox(0,0)[t]{$\bar\xi_{1,a}$}}
\put(-10,-3){\makebox(0,0)[t]{$\varpi_{1,a}$}}
}

\put(0,0){
\put(0,0){\thicklines{\qbezier(70,50)(20,50)(10,0)}}
\put(70,50){\thicklines{\line(1,0){30}}}
\put(100,47){\makebox(0,0)[tr]{$\phi_{1,b}$}}
\put(27,30){\makebox(0,0)[tl]{$\phi_{1,b}$}}
\multiput(70,0)(0,5){10}{$.$}
\put(70,-3){\makebox(0,0)[t]{$\bar\xi_{1,b}$}}
\put(10,-3){\makebox(0,0)[tl]{$\varpi_{1,b}$}}
}
}

\put(240,0){
\put(0,0){\vector(1,0){100}}
\put(0,0){\line(-1,0){100}}
\put(0,50){\line(1,0){100}}
\put(0,50){\line(-1,0){100}}
\put(0,100){\line(1,0){100}}
\put(0,100){\line(-1,0){100}}
\put(100,-3){\makebox(0,0)[tr]{$\xi$}}
\put(0,0){\vector(0,1){110}}
\put(-3,110){\makebox(0,0)[tr]{$\phi$}}
\put(0,47){\makebox(0,0)[tl]{$\rho_0$}}
\put(2,97){\makebox(0,0)[tl]{$\orho$}}

\put(0,0){
\put(0,0){\thicklines{\qbezier(-70,50)(-20,50)(-10,100)}}
\put(-70,50){\thicklines{\line(-1,0){30}}}
\put(-100,47){\makebox(0,0)[tl]{$\phi_{2,a}$}}
\put(-27,80){\makebox(0,0)[tr]{$\phi_{2,a}$}}
\multiput(-70,0)(0,5){10}{$.$}
\put(-70,-3){\makebox(0,0)[t]{$\bar\xi_{2,a}$}}
\put(-10,-3){\makebox(0,0)[t]{$\varpi_{2,a}$}}
\multiput(-10,0)(0,5){20}{$.$}
\multiput(9,0)(0,5){20}{$.$}
}

\put(0,0){
\put(0,0){\thicklines{\qbezier(70,50)(20,50)(10,100)}}
\put(70,50){\thicklines{\line(1,0){30}}}
\put(100,47){\makebox(0,0)[tr]{$\phi_{2,b}$}}
\put(27,80){\makebox(0,0)[tl]{$\phi_{2,b}$}}
\multiput(70,0)(0,5){10}{$.$}
\put(70,-3){\makebox(0,0)[t]{$\bar\xi_{2,b}$}}
\put(10,-3){\makebox(0,0)[tl]{$\varpi_{2,b}$}}
}
}
}
\end{picture}
\caption{\label{f:4phi}{The profiles $\phi_{1,a}$, $\phi_{1,b}$, $\phi_{2,a}$, $\phi_{2,b}$.}}
\end{figure}
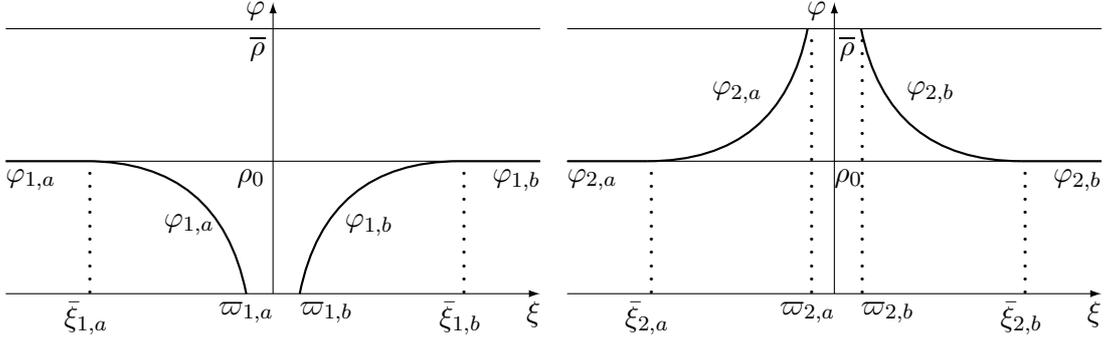

Consider now equation \eqref{e:tws} and define
\[
D_1(\rho)=D(\overline{\rho}-\rho), \quad g_1(\rho)= -g(\overline{\rho}-\rho),\quad   h_1(\rho)=h(\overline{\rho} - \rho),
\]
with $\rho\in [0,\overline{\rho} -\rho_0]$. Consider the equation 
\begin{equation}\label{eq:psi1}
\left(D_1(\psi)\psi'\right)' + \left(c-h_1(\psi)\right)\psi' +g_1(\psi)=0.
\end{equation}
Let $\psi(\xi) := \orho-\phi(\xi)$.  It is easy to see that $\phi$ satisfies \eqref{e:tws} if and only if $\psi:= \orho-\phi(\xi)$ satisfies \eqref{eq:psi1}. Notice that $\phi$ ranges in $[\rho_0,\orho]$ as long as $\psi$ ranges in $[0,\orho-\rho_0]$.  Again by \cite[Theorem 2.7]{Corli-Malaguti}, equation \eqref{eq:psi1}
has strict classical traveling-wave profile from $\orho - \rho_0$ and to $\orho - \rho_0$, for every speed $c \in \mathbb{R}$; they are also unique (up to shifts) in the class of classical or sharp traveling-wave solutions. Further, by applying condition \eqref{eq:goodg}, we obtain
\begin{equation*}
g_1(\rho)\ge L\left(\orho-\rho_0-\rho\right)^{\alpha}\quad \hbox{in a left neighborhood of $\orho-\rho_0$},
\end{equation*}
with $L$ and $\alpha$ as in \eqref{eq:goodg}. Hence, by \cite[Theorem 2.9]{Corli-Malaguti} the corresponding profiles are such that
\begin{equation}\label{eq:psi}
\left\{
\begin{array}{ll}
\hbox{$\psi_{a}$ from $\orho-\rho_0$} 
&\left\{
\begin{array}{l}
\psi_{a}(\xi) \equiv \orho-\rho_0, \ \xi\in(-\infty,\oxi_a],
\\
\psi_{a}(\xi)\in[0,\orho-\rho_0), \ \xi\in(\oxi_{a},\varpi_a],
\end{array}
\right.
\\[4mm]
\hbox{$\psi_{b}$ to $\orho-\rho_0$}
&\left\{
\begin{array}{l}
\psi_{b}(\xi)\in[0,\orho-\rho_0), \ \xi\in[\varpi_b, \oxi_b),
\\
\psi_{b}(\xi)\equiv\orho-\rho_0, \ \xi\in[\oxi_{b},\infty),
\end{array}
\right.
\end{array}
\right.
\end{equation}
for some $\oxi_{a}$ and $\oxi_{b}$.
Hence we obtain the additional semi-wavefronts profiles for \eqref{e:E}
\begin{equation}\label{eq:due}
\left\{
\begin{array}{ll}
\hbox{$\phi_{2,a}$ from $\rho_0$} 
&
\left\{
\begin{array}{l}
\phi_{2,a}(\xi) \equiv \rho_0, \ \xi\in(-\infty,\oxi_{2,a}],
\\
\phi_{2,a}(\xi)\in(\rho_0,\orho], \ \xi\in(\oxi_{2,a},\varpi_{2,a}],
\end{array}
\right.
\\
\hbox{$\phi_{2,b}$ to $\rho_0$}
&\left\{
\begin{array}{l}
\phi_{2,b}(\xi)\in(\rho_0,\orho], \ \xi\in[\varpi_{2,b}, \oxi_{2,b}), 
\\
\phi_{2,b}(\xi)\equiv\rho_0, \ \xi\in[\xi_{2,b},\infty).
\end{array}
\right.
\end{array}
\right.
\end{equation}
So we conclude the proof by pasting as follows the various semi-wavefronts. 
\begin{enumerate}[{\em (1)}]

 \item We consider the profiles $\phi_{1,b}$ and $\phi_{2,a}$ and we shift them in such a way that $\oxi_{1,b}=\oxi_{2,a}$; this is always possible. Then $\phi_1 \colon [\varpi_{1,b}, \varpi_{2,a}] \to [0, \orho]$ defined by
     \[
 \phi_1(\xi)=\left\{ \begin{array}{ll}\phi_{1,b}(\xi) & \text{if }\varpi_{1,b}\le \xi\le \oxi_{1,b}=\oxi_{2,a},\\
 \phi_{2,a}(\xi) & \text{if }\oxi_{2,a}\le \xi\le \varpi_{2,a}, \end{array}\right.
 \]
is a profile of a traveling-wave solutions which satisfies the required condition.

 \item We consider the profiles $\phi_{2,b}$ and $\phi_{1,a}$ and we shift them in such a way that $\oxi_{2,b}=\oxi_{1,a}$; again this is possible. Then $\phi_2 \colon [\varpi_{2,b}, \varpi_{1,a}] \to [0, \orho]$ is defined by
     \[
 \phi_2(\xi)=\left\{ \begin{array}{ll}\phi_{2,b}(\xi) & \text{if }\varpi_{2,b}\le \xi\le \oxi_{2,b}=\oxi_{1,a},\\
 \phi_{1,a}(\xi) & \text{if }\oxi_{1,a}\le \xi\le \varpi_{1,a}. \end{array}\right.
 \]

 \item The profile $\phi_3$, of this traveling-wave solution is obtained by pasting $\phi_{1,a}$ and $\phi_{1,b}$ in the case when $\oxi_{1,a}=\oxi_{1,b}$. In particular $\phi_3 \colon [\varpi_{1,b}, \varpi_{1,a}] \to [0, \rho_0]$ is defined by
     \[
 \phi_3(\xi)=\left\{ \begin{array}{ll}\phi_{1,b}(\xi) & \text{if }\varpi_{1,b}\le \xi\le \oxi_{1,b}=\oxi_{1,a},\\
 \phi_{1,a}(\xi) & \text{if }\oxi_{1,a}\le \xi\le \varpi_{1,a}. \end{array}\right..
 \]

 \item The profile $\phi_4$ of this traveling-wave solution is obtained by pasting $\phi_{2,a}$ and $\phi_{2,b}$ in the case when $\oxi_{2,a}=\oxi_{2,b}$. In particular $\phi_4 \colon [\varpi_{2,b}, \varpi_{2,a}] \to [\rho_0, \orho]$ is defined by
     \[
 \phi_4(\xi)=\left\{ \begin{array}{ll}\phi_{2,b}(\xi) & \text{if }\varpi_{2,b}\le \xi\le \oxi_{2,b}=\oxi_{2,a},\\
 \phi_{2,a}(\xi) & \text{if }\oxi_{2,a}\le \xi\le \varpi_{2,a}. \end{array}\right.
 \]
 \end{enumerate}
The theorem is completely proved. 
\end{proofof}


\begin{figure}[htbp]
\begin{picture}(100,120)(-80,-10)
\setlength{\unitlength}{1pt}

\put(-20,0){
\put(0,0){\vector(1,0){350}}
\put(0,0){\line(-1,0){40}}
\put(0,60){\line(1,0){350}}
\put(0,120){\line(1,0){350}}
\put(0,120){\line(-1,0){40}}
\put(0,60){\line(-1,0){40}}
\put(350,8){\makebox(0,0){$\xi$}}
\put(130,-10){\vector(0,1){160}}
\put(137,147){\makebox(0,0){$\phi$}}
\put(137,67){\makebox(0,0){$\rho_0$}}
\put(137,127){\makebox(0,0){$\orho$}}

\put(-40,0){
\put(0,0){\thicklines{\qbezier(80,120)(100,60)(150,60)}}
\put(150,60){\thicklines{\line(1,0){80}}}
\put(115,92){\makebox(0,0)[br]{$\phi_{2,b}$}}
\put(200,63){\makebox(0,0)[b]{$\phi_{2,b}$}}
\put(150,-2){\makebox(0,0)[t]{$\oxi_{2,b}$}}
\multiput(149,0)(0,5){12}{$.$}
\multiput(79,0)(0,5){24}{$.$}
\put(80,-2){\makebox(0,0)[t]{$\varpi_{2,b}$}}
}

\put(110,0){
\put(0,0){\thicklines{\qbezier(80,60)(130,60)(170,0)}}
\put(160,40){\makebox(0,0)[br]{$\phi_{1,a}$}}
\multiput(79,0)(0,5){12}{$.$}
\put(80,-2){\makebox(0,0)[t]{$\oxi_{1,a}$}}
\put(170,-2){\makebox(0,0)[t]{$\varpi_{1,a}$}}
}
}

\end{picture}
\caption{\label{f:join}{The profile $\tilde\phi_1$ in \eqref{e:phii}.}}
\end{figure}
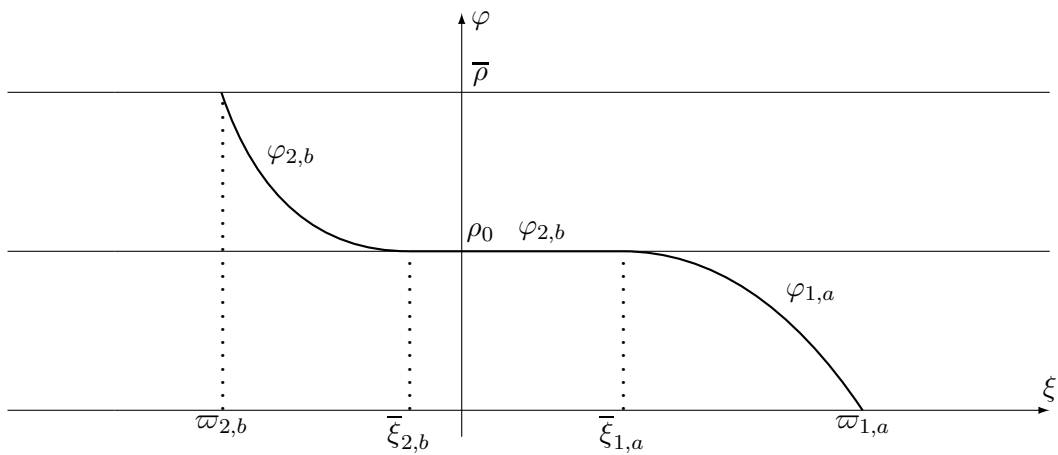

\begin{remark} The wave profiles $\phi_1$ and $\phi_2$ in Theorem \ref{t:tws} are strictly monotone because they assume the value $\rho_0$ exactly in one point. We can also construct profiles where the value $\rho_0$ is reached in an interval, hence losing the strict monotonicity. For instance, when $\oxi_{2,b}<\oxi_{1,a}$, the function $\tilde \phi_2 \colon [\varpi_{2,b}, \varpi_{1,a}] \to [0, \orho]$ defined by
\begin{equation}\label{e:phii} 
\tilde\phi_2(\xi)=\left\{ \begin{array}{ll}\phi_{2,b}(\xi) & \text{if }\varpi_{2,b}\le \xi\le \oxi_{2,b},
\\
 \rho_0 & \text{if }\oxi_{2,b}\le \xi\le \oxi_{1,a},
 \\
 \phi_{1,a}(\xi) & \text{if }\oxi_{1,a}\le \xi\le \varpi_{1,a}, \end{array}\right.
\end{equation} 
is the profile of a classical strict traveling wave for \eqref{e:E} that is not strictly monotone.
\end{remark}

We recall that an analogous \lq\lq pasting\rq\rq\ between profiles has been shown in \cite[\S 8]{Corli-Malaguti} to fail. That failure was due to the fact that $D(\phi)\phi'$ was discontinuous at the point $\oxi$ of pasting;  in turn, this was a consequence of being $\phi'(\oxi\pm)\in\R$ while $|\phi(\oxi\mp)|=\infty$. In the present case, on the contrary, we always have $\phi'(\oxi\pm)=0$, because the profiles are classical; as a consequence, $D(\phi)\phi'$ vanishes at $\oxi$.




\appendix
\section*{Acknowledgements}
A. Corli was supported by the GNAMPA 2015 project {\em Balance Laws in the Modeling
of Physical, Biological and Industrial Processes}; he also acknowledges support by the PRIN 2012 project
{\em Nonlinear Hyperbolic Partial Differential Equations, Dispersive and Transport Equations:
Theoretical and Applicative Aspects}. The authors were also supported by the Project {\em Macroscopic models of traffic flows: qualitative analysis and implementation} by the University of Modena and Reggio Emilia and are members of GNAMPA.
{\small
\bibliography{refe_sharp}
\bibliographystyle{abbrv}
}

\end{document}